\documentclass[a4paper,12pt,reqno]{amsart}
\usepackage{mymacros}
\usepackage[
    inline,
    final
]{ showlabels }

\excludecomment{comment}

\usepackage[whole]{bxcjkjatype}

\newcommand{\grothendieckgroup}{\mathsf{K _{ 0 }}}

\newcommand{\dgPerf}{\mathop{\scrP\!\mathit{erf}}}

\title{
    Nonexistence of phantom categories on very general noncommutative projective planes
}

\date{\today}

\author{K.~Murai}
\address{
    Department of Mathematics, Graduate School of Science, Osaka University, Machikaneyama 1-1, Toyonaka, Osaka, 560-0043, Japan.
}

\email{u037661k@ecs.osaka-u.ac.jp}

\begin{document}

\begin{abstract}
    We show that very general noncommutative projective planes do not admit phantom categories.
\begin{comment}
    \emph{ガイドライン}
    \begin{enumerate}[(1)]
        \item この論文の主定理の主張
        \item 証明には、Borisov-Kemboiの議論と、ATV2で確立された非可換射影平面上の連接層に関する諸性質を用いる。
    \end{enumerate}
\end{comment}
\end{abstract}

\maketitle

\tableofcontents

%
%

\section{Introduction}
Let \(X\) be a smooth projective variety. A full triangulated subcategory \(\cA\subset\dbcoh(X)\) is called \emph{admissible} if the inclusion admits both left and right adjoints.
Such categories are also called \emph{geometric noncommutative schemes} \cite{ORLOV201659}.
A nontrivial geometric noncommutative scheme is called a \emph{quasi-phantom category} if it has a finite Grothendieck group and trivial Hochschild homology, and it is called a \emph{phantom category} if, in addition, it has trivial Grothendieck group.

In \cite{kuznetsov2009hochschildhomologysemiorthogonaldecompositions}, Kuznetsov conjectured that quasi-phantom categories do not exist, and pointed out that this nonexistence could be used to prove the termination of semi-orthogonal decompositions.
However, this conjecture was disproved by the explicit construction of quasi-phantom categories on classical Godeaux surfaces \cite{MR3062745} and Burniat surfaces \cite{MR3096524}.
Subsequently, phantom categories were constructed on the products of surfaces with a quasi-phantom category \cite{MR3090263} and on determinantal Barlow surfaces \cite{MR3361723}.
These surfaces on which the phantom category was constructed are all of the non-negative Kodaira dimension.
In particular, they are non-rational.
Therefore, it is interesting to consider whether phantom categories exist on rational surfaces.

Among others, there was a folklore conjecture that phantom categories do not exist on del Pezzo surfaces, especially on \(\bP^{ 2 }\).
A del Pezzo surface is either a blow-up of \(\bP^2 \) at most eight points in a general position or \(\bP^{ 1 }\times\bP^{ 1 }\).
This conjecture was solved in the affirmative by Pirozhkov \cite{PIROZHKOV2023109046}.
Borisov and Kemboi extend this result on del Pezzo surfaces to a bigger class of rational surfaces:

\begin{theorem}[{=\cite[{Theorem~3.1}]{borisov2024nonexistencephantomsnongenericblowups}}]\label{theorem:main result of Borisov-Kemboi}
    Let \(X\) be the blow-up of \(\bP_{\bC}^2\) at a finite set of points on a smooth cubic curve \(E\) such that the restriction map 
    \(
        \picardgroup(X)
        \to
        \picardgroup(E)
    \)
    is injective.
    Then \(\dbcoh( X )\) admits no phantom categories.
\end{theorem}

The assumption of \cref{theorem:main result of Borisov-Kemboi} is satisfied if the points on \(E\) are in a very general position.
However, the assumption that all points lie on a cubic curve is a nontrivial closed condition as soon as the number of points is more than nine.
In this sense, the points as in \cref{theorem:main result of Borisov-Kemboi} are in a special position.

In contrast, Krah \cite{MR4701883} constructed phantom categories on blow-ups of \(\bP^2\) at ten points in a general position:

\begin{theorem}[{=\cite[{Theorem~1.1}]{MR4701883}}]\label{theorem:main result of Krah}
    Let \(X\) be the blow-up of \(\bP^2_{ \bC }\) in \(10\) closed points \(p_1,\dots,p_{10}\in\bP^2_{ \bC }\) in a general position. 
    Then there exists an exceptional collection \((\cL_{1},\dots, \cL_{ 13 })\) consisting of line bundles such that its right orthogonal complement 
    \(
        \langle
        \cL_{1},\dots, \cL_{ 13 }
        \rangle^{\perp}
    \)
    is a phantom category.
\end{theorem}

Thus the existence of phantom categories is a subtle problem even for rational surfaces.
Moreover, \cref{theorem:main result of Borisov-Kemboi} and \cref{theorem:main result of Krah} mean that a surface without phantom categories may deform to surfaces with phantom categories.

In this paper, we initiate the study of phantom categories in the broader context of noncommutative algebraic geometry.
More specifically, we study flat deformations of \(\coh(\bP^{2})\) as abelian categories \cite{MR2238922}.
According to \cite[{Theorem~7.2}]{MR2836401}, they can be realized as 
\emph{noncommutative projective planes}, which is defined as the quotient category 
\begin{equation*}
    \qgr(A)\coloneqq \grmod(A)/\tors(A), 
\end{equation*}
where \(A\) is a three-dimensional Artin-Schelter (AS-)regular quadratic (graded) algebra, \(\grmod(A)\) denotes the category of finitely generated graded \(A\)-modules, and \(\tors(A)\) denotes the Serre subcategory of finite-dimensional modules.
It is known that \(\boundedderived\qgr(A)\) is a geometric noncommutative scheme \cite[{Theorem~5.8}]{ORLOV201659}.

Three-dimensional quadratic AS-regular algebras are classified by triples \((E,\sigma,\cL)\), where \(E\subset\bP^2\) is a cubic divisor, \(\sigma\in\Aut(E)\) is an automorphism, and \(\cL\in\picardgroup(E)\) is a very ample invertible sheaf.
The cubic divisor \(E\) is considered to be embedded in the noncommutative projective plane via a pair of adjoint functors (see \eqref{eq:adunction of j}).

\begin{equation}\label{eq:embedding cubic divisor into noncomutative projective plane}
    \bL j^{ \ast } 
    \colon
    \boundedderived
    \qgr( A )
    \rightleftarrows
    \dbcoh(E)
    \colon
    j_{ \ast }
\end{equation}

While \(\boundedderived\coh(\bP^{2})\) has no phantom categories, this result does not immediately imply the nonexistence of phantoms on noncommutative projective planes as we pointed out above.
Nevertheless, using a noncommutative analogue of the methods of Pirozhkov \cite{PIROZHKOV2023109046} and Borisov-Kemboi \cite{borisov2024nonexistencephantomsnongenericblowups}, we obtain the following result.

\begin{theorem}[{=\cref{Main theorem}}, MAIN RESULT]\label{Main result}
    Let \(A\) be a three-dimensional AS-regular quadratic algebra associated to a geometric triple \((E,\sigma, \cL)\).
    Assume that \(E\) is nonsingular and \(\sigma\) is a translation of infinite order. Then
    \(
        \boundedderived \qgr(A)
    \)
    admits no phantom categories.
\end{theorem}

This theorem shows that very general noncommutative deformations of the projective plane do not admit phantom categories.
\begin{comment}
\begin{remark}
While phantom categories can be defined more generally for smooth proper dg categories, in this paper we restrict phantom categories to geometric noncommutative schemes to simplify the definition. Therefore, we need to verify that the derived category of the noncommutative projective plane is a geometric noncommutative scheme.
The derived category of noncommutative projective plane \(\boundedderived\qgr( A )\) has a full strong exceptional collection.
By \cite[{Theorem~5.8}]{ORLOV201659}, every triangulated category with a full exceptional collection is a geometric noncommutative scheme.
Therefore, the derived noncommutative projective plane is geometric noncommutative scheme.
This allows us to study phantom categories on noncommutative projective plane.
\end{remark}
\end{comment}
The assumptions of \cref{Main result} are rather technical, and we expect:
\begin{conjecture}
    Every noncommutative projective plane admits no phantom categories.
\end{conjecture}
Our methods do not directly extend to arbitrary noncommutative projective planes.
One reason for this limitation is that some results from \cite{ATV2} cannot be applied to arbitrary noncommutative projective planes.
Consequently, \cref{proposition:skyscraper sheaves} and \cref{proposition special point} (see below) cannot be extended to arbitrary noncommutative projective planes.

Borisov and Kemboi conjectured that any smooth projective surface which admits an effective (anti-)canonical divisor does not admit  phantom categories \cite[{Conjecture~1.3}]{borisov2024nonexistencephantomsnongenericblowups}.
This conjecture naturally extends to the noncommutative setting, especially to blow-ups at points on the anti-canonical divisor \(E\) of noncommutative projective planes \cite{MR1846352}.
Namely, we expect:
\begin{conjecture}
    Any blow-up at points on the anti-canonical divisor \(E\) of a noncommutative projective plane admits no phantom categories.
\end{conjecture}

\begin{comment}
\todo{
    非可換射影平面のblow-upはanti-canonical divisor 上の天のblow-upになるのでBorisov-kemboi型のblow-upとなり, phantom を持たないことが期待される.
}
\todo{
    NC surfaceはPoisson str の変形から得られる.
    Poisson structure の零点集合は anti-canonical divisor 
    Krah型はNC deformがない
}
\todo{主定理の意義・今後考えられる方向(7) (8)あたりのこと}
\todo{Open question: 変形族のvery general なmemberがphantomを持たなければ、すべてのmemberがphantomを持たない。}
\end{comment}

\subsection{Outline}
In \cref{subsec:admissible subcategories and phantoms} we recall fundamental definitions and results on admissible subcategories and phantom categories.
In \cref{subsec:noncommutative projective plane} we recall the definition and results about AS-regular algebras.
In \cref{sec:spherical functor} we recall the notion of spherical functors.
\cref{subsec:spanning class of noncommutative projective plane} and \cref{subsec:support of graded modules restricted to anti-canonical divisor} develop key technical tools required for our main theorem. 
The following proposition proved in \cref{subsec:spanning class of noncommutative projective plane} roughly corresponds to Step \(3\) of the proof of \cref{theorem:main result of Borisov-Kemboi}. 
\begin{proposition}[{=\cref{proposition:skyscraper sheaves on E is a spanning class}}]\label{proposition:skyscraper sheaves}
    Assume that the order of the automorphism \( \sigma \) associated with the three-dimensional AS-regular algebra \(A\) is infinite.
    Then the set 
    \(
        \{
            j_{\ast} \cO_{ p }
            \mid 
            p\in E
        \}
    \)
    is a spanning class of 
    \(
        \boundedderived\qgr ( A )
    \).
\end{proposition}
The proof relies on Artin-Tate-Van den Bergh's results on finitely generated graded modules over AS-regular algebras \cite[Propositions 7.5, 7.9]{ATV2}.
The infinite order condition on \(\sigma\) is necessary for applying these results.

The following proposition proved in \cref{subsec:support of graded modules restricted to anti-canonical divisor} roughly corresponds to a consequence of the injectivity of the restriction map in \cref{theorem:main result of Borisov-Kemboi} for \(X=\bP^{2}\) which plays an important role in the proof of \cref{theorem:main result of Borisov-Kemboi}.
\begin{proposition}[{=\cref{proposition:support of E circ}}]\label{proposition special point}
    If the automorphism \( \sigma \) has infinite order, we have 
    \(
        E
        \setminus 
        E^{\mathrm{sp}}
        \neq 
        \emptyset
    \) 
    where
    \(
        E^{ \mathrm{sp} }
        \coloneqq
        \{
            p\in E
            \mid 
            \Supp( 
                \bL j^{ \ast }  M
            )
            =\{p\}
            \text{ for some }
            M\in \boundedderived \qgr( A ) 
        \}.
    \) 
\end{proposition}
Based on these propositions, we prove the main theorem in \cref{subsec:proof of main result}.

\begin{acknowledgements*}
The author would like to express his gratitude to his advisor, Shinnosuke Okawa, for providing helpful comments during regular seminars.
\end{acknowledgements*}

\begin{comment}
\subsection{Introductionのガイドライン}
\begin{enumerate}[(1)]
    \item phantom圏が現れた歴史的経緯
    \begin{itemize}
        \item Kuznetsov予想(quasi-phantomは存在しない)\(\to\)SODは有限の長さで止まる(DK仮説を通じて、termination of flipと対応)
        \item quasi-phantomの発見(classical Godeaux、Sosna et al)
        \item phantomの発見(Gorchinskiy-Orlov, Barlow surface)
        \item no phantom on rational surface (folklore conjecture)
    \end{itemize}
    \item 有理曲面上にphantomは存在するか？という問題に関する先行研究(Pirozhkov, Krah, Borisov-Kemboi)
    \item 非可換代数幾何学と非可換射影平面について(射影平面は変形しないが、非可換変形はする。それらは、3次元Artin-Schelter regular quadratic algebraのqgrとして得られるということがわかっている(van den bergh, noncommutative quadrics))
    \item この論文の問(非可換射影平面にphantomは存在するか？)
    \begin{itemize}
        \item 変形前にphantomがなくても、変形後にphantomが存在する場合があるので(Borisov-Kemboi + Krah)、これは非自明な問題である。
    \end{itemize}
    \item この論文の主結果
    \item 証明に使う主な道具や、それが本論文のどのセクションで説明されるか？という点に関する簡単な説明
    \item この論文の証明と同様にして、非可換2次曲面(3次元AS-regular cubic Z-algebraのqmod)の場合も同様の事実が証明できるということ
    \item 残された問題、今後考えられる問題（予想）。例えば、「非可換有理曲面には絶対にphantomが存在しない」と言えるか？また、任意の非可換射影平面で本論文の主結果が成立するか？
\end{enumerate}
\end{comment}

\section{Preliminaries}
%
%
\subsection{Admissible subcategories and phantom categories}\label{subsec:admissible subcategories and phantoms}
In this section we review the basic definition (for details, see \cite{huybrechts2006fourier}).
Let \(\cT\) be a \(\bfk\)-linear triangulated category, where \(\bfk\) is a base field. 
For \(X, Y\in\cT\), we define a graded \(\bfk\)-vector space
\begin{equation*}
    \RHom_{ \cT }( X, Y )
    \coloneqq
    \bigoplus_{ \ell\in\bZ }
    \Hom_{ \cT }( X, Y[ \ell ] ).
\end{equation*}

A  triangulated category \(\cT\) is \emph{of finite type over \(\bfk\)} if \(\dim_{ \bfk }\RHom_{ \cT }( X, Y ) < \infty\) for all \(X, Y\in \cT\). 
We assume that any triangulated category is of finite type over a base field \(\bfk\) throughout this paper.

\begin{definition}
    A full triangulated subcategory \(\cA\subset\cT\) is called \emph{admissible} if the inclusion functor admits a left and right adjoint functor.
\end{definition}

\begin{definition}
    A \emph{semi-orthogonal decomposition} of \(\cT\) is an ordered pair of admissible subcategories \(( \cA,\cB )\) of \(\cT\) which satisfies the following conditions:
    \begin{enumerate}
        \item \(\RHom(b,a)=0\) for any \(a\in\cA\) and \(b\in\cB\).
        \item \(\cT\) is the smallest triangulated category which contains \(\cA\) and \(\cB\).
    \end{enumerate}
    If \((\cA, \cB)\) is a semi-orthogonal decomposition of \(\cT\), we write \(\cT=\langle\cA,\cB\rangle\).
\end{definition}

The conditions (1) and (2) imply that for any \(x\in\cT\), there exists a unique exact triangle up to isomorphism
\begin{equation*}
    b\to x\to a \to b[1]
\end{equation*}
where \(a\in\cA\) and \(b\in\cB\).
The correspondences \(x \mapsto a\) and \(x \mapsto b\) induce functors \(\mathcal{T} \to \mathcal{A}\) and \(\mathcal{T} \to \mathcal{B}\), respectively. These are called the \emph{left projection onto \(\mathcal{A}\)} and the \emph{right projection onto \(\mathcal{B}\)}.

\begin{definition}
    An object \(E\in\cT\) is called \emph{exceptional} if \(\RHom_{ \cT }( E, E ) = \bfk[0]\), i.e., \(\Hom_{ \cT }(E, E[ \ell ]) = 0\) when \(\ell\neq 0\) and \(\Hom_{ \cT }( E,E ) = \bfk\).
    An \emph{exceptional collection} in \(\cT\) is a sequence of exceptional objects \((E_{ 1 }, \dots, E_{ n })\) which satisfies \(\RHom_{ \cT }( E_{ i }, E_{ j } ) = 0\) for all \(i > j\).
\end{definition}

\begin{definition}
    An exceptional collection \((E_{ 1 }, \dots, E_{ n })\) is called \emph{strong} if in addition, \(\RHom_{ \cT }( E_{ i }, E_{ j } ) = 0\) for all \(i\neq j\).
\end{definition}

\begin{definition}
    The \emph{Grothendieck group} of \(\cT\), denoted \(\grothendieckgroup(\cT)\), is the abelian group generated by the isomorphism classes of \(\cT\), with the relations of the form 
    \(
        [ X ] - [ Y ] + [ Z ]
    \)
    for all exact triangle 
    \(
        X\to Y\to Z\to X[1]
    \).
\end{definition}

By definition of the Grothendieck group, the exact functor \(F\colon \cT_{ 1 }\to \cT_{ 2 }\) induces the group homomorphism \(\grothendieckgroup(\cT_{1})\to\grothendieckgroup(\cT_{2})\).

If we have a semi-orthogonal decomposition \(\cT = \langle\cA, \cB\rangle\), then the inclusion functors induce an isomorphism
\begin{equation*}
    \grothendieckgroup(\cT)
    \simeq 
    \grothendieckgroup( \cA )
    \oplus
    \grothendieckgroup( \cB ).
\end{equation*}

\begin{definition}
    Let \(X\) and \(Y\) be a smooth projective variety and let \(E\in\dbcoh(X\times Y)\).
    The \emph{integral functor} \(\Phi_{E}\colon \dbcoh(X)\to \dbcoh(Y)\) is defined by 
    \begin{equation*}
        \Phi_{E}( F ) = p_{\ast}(q^{\ast}(F) \Lotimes E)
    \end{equation*} 
    where \(q\colon X\times Y\to X\) and \(p\colon X\times Y\to Y\) are projections. 
    The object \(E\) is called the \emph{kernel of the integral functor} \(\Phi_{E}\).
\end{definition}

\begin{definition}[{\cite{ORLOV201659}}]
    A \emph{geometric noncommutative scheme} is an admissible subcategory of \(\dbcoh( X )\) where \(X\) is a smooth projective variety.
\end{definition}

By definition of geometric noncommutative schemes, any admissible subcategory of a geometric noncommutative scheme is also a geometric noncommutative scheme.

\begin{example}
    A triangulated category with a full exceptional collection is a geometric noncommutative scheme by 
    \cite[{Theorem~5.8}]{ORLOV201659}.
    By \cref{theorem: full strong exceptional collection on fake p2}, every noncommutative projective plane is a geometric noncommutative scheme.
\end{example}

A geometric category \(\cA\) inherits a differential graded (dg) enhancement from the derived category of coherent sheaves on some smooth projective variety.
By using this enhancement of \(\cA\), we can define the Hochschild homology \(\HH_{ \bullet }( \cA )\) of \(\cA\).

\begin{definition}
    A nontrivial geometric noncommutative scheme \(\cA\) is called a \emph{quasi-phantom category} if its Hochschild homology \(\HH_{\bullet}(\cA) = 0\) and its Grothendieck group \(\grothendieckgroup(\cA)\) is finite.
    It is a \emph{phantom category} if in addition \(\grothendieckgroup(\cA)=0\).
\end{definition}

\begin{remark}
    We omit the detailed definition of Hochschild homology because our main theorem yields a stronger result: the nonexistence of admissible subcategories with vanishing Grothendieck group, which implies the nonexistence of phantom categories (see the detailed definition of Hochschild homology in \cite{MR2275593} or \cite{kuznetsov2009hochschildhomologysemiorthogonaldecompositions}).
\end{remark}

%
%
\subsection{Noncommutative projective planes}\label{subsec:noncommutative projective plane}
\subsubsection{Three dimentional Artin-Schelter regular quadratic algebras}
Let 
\(
    \bfk
\)
be an algebraically closed field of characteristic zero.
All algebras are assumed to be over \( \bfk \) unless otherwise stated.
Let 
\(
    A
    = 
    \bigoplus_{ n \in \bZ } A_{ n }
\)
be a \( \bZ \)-graded algebra.
Let
\(
    \Grmod ( A )
\) 
denote the category of graded right modules over \( A \).
The \emph{truncation} is an endofunctor of \( \Grmod ( A ) \) defined by 
\begin{equation}\label{eq: truncation}
    \Grmod ( A )
    \to
    \Grmod ( A ),\ 
    M
    \mapsto
    M_{ \geq m}
    \coloneqq \bigoplus_{ k\geq m } M_{ k } 
\end{equation}
for 
\(
    m \in \bZ
\),
similarly, the graded modules \( M_{ > m } \), \( M_{ < m } \), and \( M_{ \leq m } \) are also defined.
The \( n \)-th \emph{shift} of gradings for \( n \in \bZ \) is an autoequivalence of \( \Grmod ( A ) \) defined by 
\begin{equation}\label{eq: shift}
    \Grmod ( A )
    \to
    \Grmod ( A ),\ 
    M\mapsto M( n ).
\end{equation}
The graded module \( M( n ) \) is defined by 
\(
    M( n )_{ k }
    =
    M_{ n + k }.
\)
We say that \( A \) is \emph{positively graded} if 
\(
    A_{ < 0 }
    =
    0
\).
For a positively graded algebra \( A \),
the 
\(
    A_{ > 0 }
\)
is a two-sided ideal, 
and there exists a canonical isomorphism of graded bimodules
\(
    A_{ 0 }
    \simeq 
    A / A_{ >0 }
\).

\begin{definition}\label{def: d-dimensional AS-regular algebra}
    Let \(A\) be a positively \(\bZ\)-graded algebra. We say \(A\) is \emph{\(d\)-dimensional Artin-Schelter (AS-)regular} if the following conditions are satisfied.
    \begin{enumerate}
        \item \( A \) is connected. i.e., \(  A_0 = \bfk \),
        \item \( \gldim A = d \),
        \item \( \dim A_{n} \) is bounded by a polynomial in \(n\),
        \item (Gorenstein condition) 
        \(
            \Ext^{ i }_{ A }( \bfk, A ) 
            \simeq 
            \begin{cases}
                \bfk &\text{ if } i=d,\\
                0 &\text{ if } i\neq d.\\
            \end{cases}
        \)
    \end{enumerate}
\end{definition}

Any three-dimensional AS-regular algebra generated in degree one is either a quadratic or a cubic algebra (\cite[Theorem~1.5]{ARTIN1987171}).
In this paper, we focus on three-dimensional quadratic AS-regular algebras defined as follows.

\begin{definition}\label{def: three dimensional quadratic AS-regular algebra}
    Let \( A \) be an AS-regular algebra. We say that \( A \) is a \emph{three-dimensional quadratic (AS-)regular algebra} if the minimal resolution of 
    \(
        \bfk_{ A }
    \) 
    has the form 
    \begin{equation}\label{eq: minimal resolution of the simple module}
        0
        \to 
        A(-3)
        \to
        A(-2)^{\oplus 3}
        \to
        A(-1)^{\oplus 3}
        \to
        A
        \to 
        \bfk_{ A }
        \to
        0.
    \end{equation}
\end{definition}

Let 
\( 
    \grmod ( A )
\)
be an abelian category of finitely generated graded \( A \)-modules,
and let 
\( 
    \tors ( A )
\) 
be a Serre subcategory of 
\( 
    \grmod ( A )
\) 
whose objects are finite dimensional modules.
The Serre quotient of 
\( 
    \grmod( A )
\) 
by 
\( 
    \tors( A )
\) 
will be denoted by 
\begin{equation}\label{eq: quotient qgr}
    \pi
    \colon
    \grmod ( A ) 
    \to
    \qgr ( A ) 
    \coloneqq
    \grmod ( A ) / \tors ( A ).
\end{equation}
The \(\pi\) is an exact functor.
The truncation \eqref{eq: truncation} and the shift of gradings \eqref{eq: shift} send any object in 
\(
    \tors( A )
\)
to itself.
Hence these endofunctors of 
\( 
    \grmod ( A )
\) 
induce the endofunctors of 
\( 
    \qgr ( A )
\).
We use the same notation for these endofunctors of 
\( 
    \qgr ( A )
\).

\subsubsection{Geometric algebras}

Let 
\( 
    ( X, \sigma, \cL )
\) 
be a triple of a projective variety \( X \), 
an automorphism 
\(
    \sigma \in \Aut( X )
\), 
and a very ample invertible sheaf 
\(
    \cL \in \picardgroup( X )
\).
Set 
\( 
    V= H^{ 0 } ( X, \cL )
\) 
and consider 
\( 
    X 
\) 
is embedded into 
\( 
    \bP ( V^{ \ast } ) 
\) 
by \( \cL \). 
\begin{definition}\label{def: geometric algebra}
The \emph{geometric algebra} \( A( X,\sigma, \cL ) \) associated to the triple
\( 
    ( X, \sigma, \cL ) 
\)
as above is the quotient algebra of the tensor algebra \( T( V ) \) by the two-sided ideal generated by 
\( 
    R =
    \{ 
        f \in V \otimes V 
        \mid 
        f (p, \sigma( p ) ) 
        = 
        0 
        \text{ for any point } p \in X
    \}
\).
\end{definition}

\begin{theorem}[{\cite{ATV1}}]\label{theorem: AS quads are isom to geometric algebra}
    For a three-dimensional AS-regular quadratic algebra \( A \), 
    there exists a geometric triple 
    \(
        (E, \sigma, \cL)
    \) 
    of the following types whose geometric algebra \( A( E, \sigma, \cL ) \) is isomorphic to \( A \).
    \begin{enumerate}
        \item \( E = \bP^{ 2 } \), \( \sigma\in\Aut( E ) \) and \(\cL=\cO_{ \bP^{2} } (1)\), or 
        \item \( E \subset \bP^{ 2 }\) is a divisor of degree \(3\), \(\sigma \in \Aut(E)\) and \(\cL=\cO_{ E }( 1 ) \), such that \(\sigma^{ \ast }\cL \ncong \cL\), \((\sigma^{2}) ^ { \ast } \cL \otimes_{ E } \cL \simeq \sigma^{ \ast } \cL \otimes_{ E } \sigma^{ \ast } \cL\).
    \end{enumerate}
    Conversely, a geometric algebra associated to a triple \((E, \sigma, \cL)\) as above is a three-dimensional quadratic AS-regular algebra.
\end{theorem}

Let \( ( E, \sigma, \cL ) \) be a geometric triple of type \( (2) \) as in the above theorem.
Then the \emph{twisted homogeneous coordinate ring} associated to the triple \( (E, \sigma , \cL) \) is a \(\bZ\)-graded algebra \(B( E, \sigma, \cL )\) which is constructed in the following way.
\begin{equation*}
    B( E, \sigma, \cL ) = \bigoplus_{n} B_{ n }
\end{equation*}
where 
\(
    B_{ n }
    =
    H^{ 0 } ( E, \cL_{ n })
\),
\(
    \cL_{ n }
    =
    \cL\otimes \cL^{ \sigma } \otimes \cdots \otimes \cL^{ \sigma^{ n - 1 } }
\)
and 
\(
    \cL^{ \sigma^{ \ell } }
    =
    (\sigma^{ \ell })^{ \ast } \cL
\).
The multiplication is defined by 
\begin{align*}
    B_{ n }
    \otimes 
    B_{ m }
    &
    =
    H^{ 0 } ( E, \cL_{ n } )
    \otimes 
    H^{ 0 } ( E, \cL_{ m } )
    \\
    &
    \simeq
    H^{ 0 } ( E, \cL_{ n } )
    \otimes 
    H^{ 0 } ( E, \cL_{ m }^{ \sigma^{ n } } )
    \overset{\otimes}{\rightarrow}
    H^{ 0 } ( E, \cL_{ n+ m } )
    =
    B_{ n+m }.
\end{align*}
The isomorphism in the second row is obtained by 
\(
    \sigma^{ n\ast }
    \colon 
    H^{ 0 } ( E, \cL_{ m } )
    \simeq 
    H^{ 0 } ( E, \cL_{ m }^{ \sigma^{ n } } ).
\)
By the definition of the algebras \( A( E, \sigma, \cL ) \) and \( B(E, \sigma, \cL) \), 
there is a canonical map 
\( 
    A( E, \sigma, \cL )
    \rightarrow
    B( E, \sigma, \cL )
\).

\begin{proposition}\label{proposition: fundamental results for AS regular algebra}
    \begin{enumerate}
        \item The canonical map \( A(E, \sigma, \cL) \to B( E, \sigma, \cL ) \) is surjective.
        \item The kernel of the canonical map is generated by an element \( g \) of degree three in the center of \( A( E, \sigma, \cL ) \) and \(g\) is unique up to scalar multiplication.
        \item There exists an equivalence 
                \begin{align}
                    \Gamma_{ \ast }
                    \colon 
                    \coh ( E )
                    &
                    \to
                    \qgr ( B( E, \sigma , \cL ) )
                    \label{eq:eqiuvalence of coh and qgr}
                    \\
                    \cE
                    &
                    \mapsto
                    \bigoplus_{ n\in \bZ } 
                    H^{ 0 }
                    \left(
                        E, \cE \otimes\bigotimes_{ i = 0 }^{ n - 1 } \sigma^{ i\ast } \cL 
                    \right).\label{eq:explicite description of equivalence of qgr and coh}
                \end{align}
    \end{enumerate}
\end{proposition}
\begin{proof}
    See \cite[{Theorem~2}]{ATV1} and \cite[{Theorem~1.3}]{MR1067406}.
\end{proof}
For simplicity, we write 
\(
    A( E, \sigma, \cL )
\) 
and 
\(
    B(E, \sigma, \cL)
\) 
as 
\( 
    A 
\) 
and 
\( 
    B 
\) 
respectively.

\begin{definition}\label{def:right multiplication of g}
    Let \(g\) be a central element of 
    \( 
        A 
    \) 
    as in \cref{proposition: fundamental results for AS regular algebra} and let 
    \(
        M
    \) 
    be a graded 
    \(
        A
    \)-modules.
    The morphism of graded 
    \( 
        A 
    \)-module 
    \begin{equation}\label{eq:right multiplication of g}
        \cdot g 
        \colon 
        M( -3 )
        \to 
        M
    \end{equation}
    is defined by 
    \(
        m
        \mapsto 
        mg
    \).
\end{definition}

\begin{definition}\label{def:the module K n}
    Let \(M\) be a graded \(A\)-module.
    The graded \(A\)-module \(K_{ M, n }\) is defined by 
    \begin{equation}
        K_{M, n }
        \coloneqq
        \ker
        \left(
            g^{n}
            \colon
            M
            \to
            M( 3n )
        \right)
    \end{equation}
    for \(n\in \bZ_{>0}\). We also define \(K_{ M, \infty }\coloneqq \bigcup_{n\in\bZ_{>0}} K_{ M, n }\).
    If \(M\) is clear, we simply write \(K_{ n }\) and \(K_{ \infty }\) instead of \(K_{ M, n }\) and \(K_{ M, \infty }\) respectively.
\end{definition}

By \cref{proposition: fundamental results for AS regular algebra}, 
we have \( A/g \simeq B \).
There exists a pair of adjoint functors
\begin{equation*}
    j^{ \ast } 
    \colon
    \qgr( A )
    \rightleftarrows
    \qgr( B )
    \overset{\eqref{eq:eqiuvalence of coh and qgr}}
    {\simeq} 
    \coh(E)
    \colon
    j_{ \ast }
\end{equation*}
where \( j^{ \ast } \) corresponds to the extension of scalars 
\(
    M 
    \mapsto 
    M
    \otimes_{ A } {}_{ A } B_{ B } 
\) 
and 
\( 
    j_{ \ast }
\) 
corresponds to the restriction of scalars.
The functor 
\( 
    j_{ \ast } 
\) 
is fully faithful and exact.
Moreover, we have a pair of adjoint functors
\begin{equation}\label{eq:adunction of j}
    \bL j^{ \ast } 
    \colon
    \boundedderived
    \qgr( A )
    \rightleftarrows
    \boundedderived
    \qgr( B )
    \simeq 
    \dbcoh(E)
    \colon
    \bR j_{ \ast } = j_{ \ast }.
\end{equation}
By \eqref{eq:adunction of j}, we have an exact triangle 
\begin{equation}
    M(-3)
    \xrightarrow{\cdot g}
    M
    \xrightarrow{\eta_{ M }}
    j_{ \ast } \bL j^{ \ast } M
    \to
    M( -3 )[ 1 ]
\end{equation}
where 
\(
    M
    \in
    \boundedderived \qgr( A )
\)
and \( \eta \) is the unit of the adjunction \eqref{eq:adunction of j}.
By the long exact sequence of this triangle, we obtain the following lemma.
\begin{lemma}\label{lemma:derived pullback of graded modules}
    Let \(M\in \grmod( A )\). Then we have 
    \begin{enumerate}
        \item \( \bL^{ 1 } j^{\ast } M \simeq j^{ \ast } \left( K_{ M, 1 }( -3 ) \right) \),
        \item \( \bL^{ k }j^{\ast} M = 0 \) if \(k\neq 0, 1\).
    \end{enumerate}
    In particular, if \(M\) has no \(g\)-torsion element, \(\bL^{1}j^{\ast} M = 0\).
\end{lemma}

The direct computation shows that:
\begin{lemma}\label{lemma: explicit desdription of grade shift on coh E}
    The shift of gradings \eqref{eq: shift} is compatible with the autoequivalence of \(\coh ( E ) \) 
    \begin{equation}\label{eq:shift of grading on coh E}
        \sigma_{ \ast } 
        \left(
            -\otimes \cL
        \right)
    \end{equation}
    via the equivalence 
    \(
        \coh( E )
        \simeq 
        \qgr( B )
    \) 
    as in \eqref{eq:eqiuvalence of coh and qgr}. i.e., there exists an isomorphism
    \begin{equation}
        \Gamma_{ \ast } ( \cE ) ( 1 )
        \simeq 
        \Gamma_{ \ast } ( \sigma_{ \ast }( \cE\otimes \cL ) )
    \end{equation}
    where 
    \(
        \cE \in \coh( E )
    \).
    In particular, we have an isomorphism
    \begin{equation}\label{eq:action of shift of gradings on skyscraper sheaves}
        \left(
            j_{ \ast }\cO_{ p }
        \right)
        ( n )
        \simeq
        j_{ \ast }
        \left(
            \cO_{ \sigma^{ n }( p ) }
        \right)
    \end{equation}
    for any \( p \in E\) and \(n\in \bZ\).
\end{lemma}
Let us introduce the explicit description of the Serre functor of 
\(
    \boundedderived \qgr( A )
\) 
for a very general three-dimensional AS-regular quadratic algebra \(A\).

\begin{theorem}\label{theorem:Serre functor of qgr(A)}
    If the automorphism \(\sigma\) is a translation and \(E\) is a smooth elliptic curve 
    then the Serre functor \( S_{ A } \) of \( \boundedderived\qgr(A) \) is of the form
    \begin{equation*}
        S_{ A } ( M )
        =
        M (-3) [ 2 ].
    \end{equation*}
\end{theorem}
\begin{proof}
    See 
    \cite[Corollaries~9.3, 9.4]{VANDENBERGH1997662}
    and 
    \cite[Theorem~A.4]{denaeghel2005idealclassesdimensionalsklyanin}.
\end{proof}

Moreover, it is known that there exists a full exceptional collection which is a noncommutative analogue of Beilinson's theorem, which addresses the case of \(\bP^{2}\).

\begin{theorem}\label{theorem: full strong exceptional collection on fake p2}
    Let 
    \( 
        \cO ( n )
        \coloneqq
        \pi ( A( n ) )
    \).
    The sequence 
    \(
        (
            \cO, \cO(1), \cO(2)
        )
    \) 
    is a full strong exceptional collection in 
    \(
        \boundedderived \qgr( A )
    \).
    In particular, 
    \(
        \boundedderived \qgr( A )
    \)
    has a tilting object.
\end{theorem}
\begin{proof}
    See \cite[{Theorem~7.1}]{abdelgadir2014compactmodulinoncommutativeprojective}.
\end{proof}

From this theorem, there exists an isomorphism
\begin{equation}\label{eq:grothendieckgroup of A}
    \grothendieckgroup( A )
    \coloneqq
    \grothendieckgroup( \boundedderived\qgr( A )  )
    \simeq 
    \bZ^{ 3 }.
\end{equation}

%
%
\subsubsection{GK dimension}

\begin{definition}\label{def:Hilbert series}
    Let \( M \) be a finitely generated graded \( A \)-module. A \emph{Hilbert series} 
    \( 
        h_{ M }( t ) 
    \) 
    is defined by 
    \begin{equation*}
        h_M(t) = \sum_{n\in\bZ} \dim(M_n) t^n \in \bZ[[t]][t^{-1}].
    \end{equation*}
\end{definition}

\begin{example}\label{example: Hilbert series of A}
    The Hilbert series of \( A \) is given by 
    \(
        h_{ A } ( t ) 
        = 
        1/ (1-t)^{ 3 }
    \)
    because of the minimal resolution \eqref{eq: minimal resolution of the simple module}.
\end{example}

The assumption
\(
    \gldim A = 3
\)
implies that there exists a projective resolution 
\begin{equation*}\label{eq:projective resolution of graded module}
    0
    \to 
    P_{ 3 }
    \to
    P_{ 2 }
    \to
    P_{ 1 }
    \to 
    P_{ 0 }
    \to 
    M
    \to
    0,
\end{equation*}
so that
\begin{align*}
    h _{ M } ( t )
    =
    \sum
    _{
        i = 0
    }^{
        3
    }
    ( - 1 ) ^{ i }
    h _{
        P _{ i }
    }
    ( t ).
\end{align*}
Since an indecomposable projective object of
\(
   \grmod A
\)
is isomorphic to
\(
   A ( \ell )
\)
for some
\(
   \ell \in \bZ
\)
(see \cite[40]{ATV1}, \cite[339]{ATV2}), combined with~\cref{example: Hilbert series of A} this implies that
\begin{equation}\label{eq: explicit description of Hilbert series of graded modules}
    h_{M} ( t )
    =
    \frac{ r }{ ( 1 - t )^{ 3 } } 
    +
    \frac{ a }{ (1 - t )^{ 2 } }
    +
    \frac{ b }{1 - t}
    +
    f(t)
\end{equation}
for uniquely determined
\(  
    a,
    b,
    r
    \in
    \bZ
\)
and
\( 
    f( t ) \in \bZ [ t^{ \pm } ] 
\).
This immediately implies:

\begin{lemma}\label{lemma:polynomial}
    For any
    \(
       M \in \grmod A
    \)
    there exists a unique polynomial
    \(
        P_{ M } ( x )
        \in
        \bQ [ x ]
    \)
    such that
    \( 
        P_{ M } ( d ) 
        = 
        \dim M_{ d }
    \) 
    for sufficient large \( d \in \bZ \).
\end{lemma}

\begin{definition}\label{def: GKdimension}
    The \emph{Gelfand-Kirillov dimension} (or the \emph{GK-dimension}) of a nontrivial graded \( A \)-module \( M \) is a pole order of Hilbert polynomial 
    \(
        h_{ M }( t )
    \) 
    at 
    \(
        t = 1
    \).
    Let
    \(
        \GKdim M 
    \)
    denote the GK-dimension of a graded module 
    \( 
        M 
    \).
\end{definition}
For
\(
    0 \neq M \in \grmod A
\)
we have the following quadchotomy by~\eqref{eq: explicit description of Hilbert series of graded modules}.
\begin{equation}\label{explicit description of GK dimension}
    \GKdim M 
    = 
    \begin{cases}
        3 & \text{ if } r > 0 \\
        2 & \text{ if } r=0 \text{ and } a > 0 \\
        1 & \text{ if } r = a = 0 \text{ and } b > 0\\
        0 & \text{otherwise}
    \end{cases}
\end{equation}
We also have
\begin{equation*}\label{eq: relationship between Gk dimension and degree of Poincare polynomial}
    \GKdim M
    =
    \deg P_{ M } ( x )
    +
    1
\end{equation*}
for the polynomial
\(
   P _{ M } ( x )
\)
as in~\cref{lemma:polynomial}. 
The degree of the zero polynomial is defined to be \( -1 \).
\begin{definition}\label{def:graded ring Lamda}
    Let us denote the localization of graded algebra \( A \) with respect to the homogeneous element \(g\) by
    \(
        \Lambda 
        \coloneqq 
        A [ g^{ -1 } ] 
    \),
    and by
    \(
        \Lambda_{ 0 }
        \subseteq
        \Lambda
    \)
    its degree
    \(
       0
    \)
    part.
\end{definition}

\begin{proposition}[{\cite[Proposition~7.5]{ATV2}}]
    The following categories are equivalent:
    \begin{enumerate}
        \item finite dimensional \( \Lambda_{ 0 } \)-modules \( V \),
        \item finitely generated graded \( A \)-modules \( N \) such that \( \dim N_{ n } \) is bounded, modulo \( g \)-torsion modules.
    \end{enumerate}
\end{proposition}

\begin{proposition}[{\cite[Corollary~7.9]{ATV2}}]\label{proposition:ATV2 Corollary7.9}
    If the order of the automorphism \( \sigma \) associated to \( A \) is infinite, then
    \( 
        \Lambda_{ 0 } 
    \) 
    has no nontrivial finite dimensional representation.
    In particular, any finitely generated graded \( A \)-module \( N \) whose dimension 
    \( 
        \dim N_{ n } 
    \) 
    is bounded is a \( g \)-torsion module.
\end{proposition}

According to the two results above, any graded \( A \)-module \( M \) with 
\(
    \GKdim M \leq 1
\) 
is a \( g \)-torsion module if the order of \(\sigma\) is infinite. 

%
%
\subsection{Spherical functor}\label{sec:spherical functor}
\begin{comment}
\cite[{Example~3.4}]{PIROZHKOV2023109046}
\end{comment}
In this section we assume that any triangulated category and any exact functor have dg enhancements.
Hence, there exist functorial cones. 
Let 
\(
    F
    \colon 
    \cT_{ 1 }
    \to 
    \cT_{ 2 }
\) 
be an exact functor admitting right and left adjoint functor 
\(
    R, L
    \colon 
    \cT_{ 2 }
    \to 
    \cT_{ 1 }
\).
Consider the canonical triangles 
\begin{equation*}
    FR
    \overset{\epsilon}{\to}
    \id
    \to
    T
    \to
    FR[ 1 ]
\end{equation*}
and 
\begin{equation}\label{eq: exact triangle associaed to adjunction}
    \id
    \overset{\eta}{\to}
    RF
    \to
    C
    \to
    \id[1]
\end{equation}
where 
\(
    \eta
    \colon 
    \id_{ \cT_{ 2 } } 
    \to 
    RF
\) 
is the unit, and 
\(
    \epsilon
    \colon 
    FR
    \to 
    \id_{ \cT_{ 1 } }
\) 
is the counit of adjunction.
The functor \(T\) is called \emph{twist} and \(C\) is a \emph{cotwist} of \(F\).

\begin{definition}\label{def: spherical functor}
    An exact functor 
    \(
        F
        \colon 
        \cT_1
        \to 
        \cT_2
    \) 
    which admits a left adjoint \(L\) and right adjoint \(R\)
    is called \emph{spherical} if the cotwist \( C \) is an equivalence and \(R\simeq CL\).
\end{definition}

\begin{theorem}[{\cite[{Theorem~2.3}]{Add16}}]
    For a spherical functor \(F\), the twist \(T\) is an equivalence.
\end{theorem}

\begin{example}\label{example: restriction functor is a spherical functor}
    Let \(A\) be a three-dimensional quadratic AS-regular algebra and let \((E, \sigma, \cL)\) be the triple associated to \(A\).
    Assume that \(E\) is a smooth elliptic curve and \(\sigma\) is a translation.
    The restriction functor
    \(
        \bL j^{\ast} 
        \colon
        \boundedderived\qgr(A)
        \to
        \dbcoh(E)
    \)
    is spherical. 
    Indeed, the right adjoint functor to 
    \(
        \bL j^{\ast}
    \) 
    is the functor 
    \(
        j_{\ast}
    \), 
    and the left adjoint functor 
    \(
        j_{!}
    \) 
    is defined by 
    \begin{equation*}
        j_{!}(-)
        =
        S_{A}^{-1}j_{\ast} S_{E}
        \overset{\text{\cref{theorem:Serre functor of qgr(A)}}}{=}
        \left(
            j_{ \ast }( ( - ) \otimes_{E} \omega_{E}) 
        \right) ( 3 ) [-1] 
        =
        \left(
            j_{ \ast }( - ) 
        \right) ( 3 ) [-1]
    \end{equation*}
    where \( S_{A} \), \( S_{E} \) are the Serre functors of 
    \(
        \boundedderived\qgr(A)
    \), 
    \(
        \dbcoh(E)
    \) 
    respectively.
    Considering the canonical exact triangle
    \begin{equation*}
        M( -3 )
        \xrightarrow{\cdot g}
        M
        \to
        j_{ \ast } \bL j^{ \ast } M
        \to 
        M( -3 )[ 1 ],
    \end{equation*}
    the spherical cotwist functor
    \(
        C
        \colon
        M
        \mapsto 
        M(-3)[1]
    \)
    is an autoequivalence. 
    It is easy to see that 
    \(
        Cj_{!}( M ) 
        \simeq
        j_{ \ast } (M) 
    \).
    Therefore, 
    \(
        \bL j^{ \ast }
    \)
    is a spherical functor.
\end{example}

\begin{theorem}[{\cite[{Proposition~2.1}]{Add16}}]
    Let 
    \(
        F
        \colon
        \cT_1
        \to
        \cT_2
    \)
    be a spherical functor such that the spherical cotwist is isomorphic to the Serre functor of \(\cT_1\) up to a shift.
    If 
    \(
        \cA\subset \cT_1
    \) 
    is an admissible subcategory, then the composition 
    \(
        \cA
        \hookrightarrow
        \cT_1 
        \to 
        \cT_2 
    \) 
    is also a spherical functor.
\end{theorem}

By \cref{example: restriction functor is a spherical functor}, we can apply this theorem to the restriction functor \(\bL j^{\ast}\).

\begin{corollary}\label{corollary: an exact triangle associated to the derived pullback which is a spherical functor}
    Let 
    \( 
        \cB 
        \subset 
        \boundedderived \qgr(A) 
    \) 
    be an admissible subcategory and 
    \(
        \pr^{ R }_{ \cB }
    \) 
    be the right projection onto \( \cB \).
    There exists an exact triangle in \( \dbcoh( E \times E ) \)
    \begin{equation}\label{eq:exact triangle of kernels}
        K_{ LR }
        \to
        \cO_{ \Delta }
        \to
        K_{ T }
        \to
        K_{ LR }[ 1 ]
    \end{equation}
    such that 
    \(
        \Phi_{ K_{ LR } }
        \simeq 
        \bL j^{ \ast } \pr^{ R }_{ \cB } j_{ \ast }
    \)
    and 
    \(
        \Phi_{ K_{ T } }
        \simeq 
        T
    \)
    where \( T \) is a spherical twist functor associated to the spherical functor 
    \(
        \cB
        \hookrightarrow 
        \boundedderived \qgr( A )
        \xrightarrow{\bL j^{ \ast }}
        \dbcoh(E)
    \).
    Thus for any \( F \in \dbcoh( E ) \), there exists an exact triangle
    \begin{equation}\label{eq:exact triangle associated to spherical functor}
        \bL j^{ \ast } \pr^{ R }_{ \cB } (j_{ \ast } F)
        \to
        F
        \to
        T( F )
        \to
        \bL j^{ \ast } \pr^{ R }_{ \cB } ( j_{ \ast } F )[ 1 ]
    \end{equation}
    in \( \dbcoh( E ) \).
\end{corollary}
We defer the proof of 
\cref{corollary: an exact triangle associated to the derived pullback which is a spherical functor} 
to 
\cref{section: proof of corollary}.

\section{ Proof of Maim theorem}\label{sec:proof of main theorem}
\subsection{A spanning class of noncommutative projective planes}\label{subsec:spanning class of noncommutative projective plane}
In this section, let \(A\) be a three-dimensional AS-regular algebra and let \((E, \sigma, \cL)\) be a geometric triple associated to \(A\). Assume that \(E\) is a smooth elliptic curve.
\begin{lemma}\label{lemma:inequality of gkdimension}
    Let 
    \(
        M \in \grmod A
    \) 
    and set
    \(
        Q
        \coloneqq 
        M/Mg
        =
        \coker(
            M( -3 )
            \overset{ \cdot g }{ \to } 
            M
        )
    \).
    Then the following inequality holds.
    \begin{equation}
        \GKdim Q \geq \GKdim M -1
    \end{equation}
\end{lemma}
\begin{proof}
    The exact sequence
    \begin{equation*}
        0
        \to
        K
        \to
        M( -3 )
        \xrightarrow{\cdot g}
        M
        \to 
        Q
        \to 
        0
    \end{equation*}
    implies
    \begin{equation}\label{eq: Hilbert series relation}
        h_{ Q } ( t ) = ( 1-t^{ 3 } )h_{ M } ( t ) + h_{ K } ( t ).
    \end{equation}
    By~\cref{lemma:polynomial} there are polynomials \( P_{ M } ( x ) \) and \( P_{ Q }( x ) \) such that
    \begin{equation*}
        \dim M_{ d } = P_{ M }( d ) \text{ and } \dim Q_{ d } = P_{ Q }( d )
    \end{equation*}
    for sufficiently large
    \(
       d 
    \).
    Moreover, we have 
    \begin{equation}\label{eq:gkdim equals degree of poincare polynomial}
        \GKdim M = \deg P_{ M }( x ) + 1
    \text{ and }
        \GKdim Q = \deg P_{ Q }( x ) + 1.
    \end{equation}
    By \eqref{eq: Hilbert series relation}, we have an inequality
    \begin{equation}\label{eq: an inequality of the value of Poincere polynomials}
        P_{ Q }( d )
        \geq 
        P_{ M }( d ) - P_{ M }( d-3 )
    \end{equation}
    for sufficiently large \( d \). Set 
    \( 
        \Delta_{ M } ( x )
        \coloneqq 
        P_{ M } ( x ) - P_{ M }( x - 3 )
    \).
    The leading coefficients of \( P_{ Q }( x ) \) and \( \Delta_{ M }( x ) \) are positive.
    Indeed, since 
    \( 
        P_{Q} ( d ) 
        = 
        \dim Q_{ d } 
        \geq 
        0
    \) 
    for sufficiently large \( d \), 
    the leading coefficient of 
    \( 
        P_{ Q } ( x )
    \) 
    is positive.
    Let \( \alpha \) be the leading coefficient of \( P_{ Q }( x )\).
    Since
    \begin{equation*}
        x^{ n }
        -
        ( x -3 )^{ n }
        =
        3 x^{ n-1 }
        +
        \text{ (lower degree terms)},
    \end{equation*}
    the leading coefficient of \( \Delta_{ M } ( x ) \) is equal to \( 3\alpha > 0 \).
    
    Since the leading coefficients of 
    \( 
        P_{ Q }( x ) 
    \) 
    and 
    \( 
        \Delta_{ M }( x ) 
    \) 
    are positive, 
    it follows from 
    \eqref{eq: an inequality of the value of Poincere polynomials} 
    that
    \begin{equation*}
        \deg P_{ Q } ( x )
        \geq
        \deg \Delta_{ M }( x )
        =
        \deg P_{ M }( x )
        -
        1.
    \end{equation*}
    Therefore,
    by \eqref{eq:gkdim equals degree of poincare polynomial},
    we have 
    \(
        \GKdim Q 
        \geq 
        \GKdim M - 1
    \).
\end{proof}

\begin{lemma}\label{lemma:support has non trivial intersection with elliptic curve}
    Suppose that the order of the automorphism
    \(
       \sigma
    \)
    is infinite. Let 
    \(
        M
    \)
    be a bounded complex in \( \qgr( A ) \).
    If 
    \(
        \bL j^{\ast} M = 0 
    \), 
    then 
    \(
        M = 0
    \).
\end{lemma}
\begin{proof}
    Consider the canonical triangle in \(\boundedderived\qgr( A )\)
    \begin{equation*}
        M( -3 )
        \xrightarrow{\cdot g}
        M
        \to
        j_{ \ast } \bL j^{ \ast } M = 0
        \to
        M( -3 )[ 1 ].
    \end{equation*}
    Thus for any \( i\in\bZ \), the morphism induced on the \(i\)-th cohomology
    \begin{equation*}
        H^{ i }( M )( -3 )
        \xrightarrow{\cdot g}
        H^{ i }( M )
    \end{equation*}
    is an isomorphism in \( \qgr(A) \).
    Hence, the cokernel 
    \(
        H^{ i }( M ) / H^{ i }( M )g
    \) 
    is a finite dimensional graded \( A \)-module, that is its GK-dimension is zero.
    By \cref{lemma:inequality of gkdimension}, we have
    \begin{equation*}
        1 \geq \GKdim H^i(M)
    \end{equation*}
    for any \(i\in\bZ\) and hence, 
    \( 
        \dim H^{ i } ( M )_{ n } 
    \) 
    is bounded.
    By \cref{proposition:ATV2 Corollary7.9}, 
    \( 
        H^{ i } ( M ) 
    \) 
    is a \(g\)-torsion module.
    Since the multiplication map 
    \(
        [ \cdot g ]
    \) 
    on 
    \(
        H^{ i } ( M ) 
    \) 
    defined by 
    \eqref{eq:right multiplication of g} 
    is an isomorphism in 
    \(
        \qgr( A ) 
    \), 
    we have 
    \( 
        H^{ i } ( M ) = 0
    \) 
    in \( \qgr( A ) \) for any \( i \in \bZ \).
    So, \( M = 0 \) in 
    \(
        \boundedderived \qgr( A )
    \).
\end{proof}

This immediately implies the following proposition:

\begin{proposition}\label{proposition:skyscraper sheaves on E is a spanning class}
    Suppose that the order of automorphism of \( \sigma \) is infinite.
    Then the set 
    \(
        \{
            j_{\ast} \cO_{ p }
            \mid 
            p\in E
        \}
    \)
    is a spanning class of 
    \(
        \boundedderived\qgr ( A )
    \).
\end{proposition}
\begin{proof}
    Let 
    \(
        M
        \in
        \boundedderived\qgr( A )
    \), 
    and assume that 
    \(
        \RHom(
             M, j_{ \ast } \cO_{ p }
        ) 
        = 
        0 
    \)
    for any \( p \in E\).
    Then, by the adjunction 
    \(
        \bL j^{ \ast } \dashv j_{ \ast }
    \), 
    we have 
    \(
        \RHom( \bL j^{ \ast } M , \cO_{ p } ) 
        = 
        0
    \)
    for any \( p \in E \). 
    Since the set 
    \( 
        \{ \cO_{ p } \}_{ p\in E } 
    \)
    is a spanning class of 
    \(
        \dbcoh ( E )
    \), 
    we have 
    \(
        \bL j^{ \ast } M 
        = 
        0
    \),
    and hence 
    \(
        M = 0
    \)
    by \cref{lemma:support has non trivial intersection with elliptic curve}.
    Therefore, the set 
    \(
        \{
            j_{ \ast } \cO_{ p }
        \}_{ 
            p\in E 
        }
    \) 
    is a spanning class of 
    \(
    \boundedderived \qgr ( A )
    \).
\end{proof}

%
%
\subsection{Support of graded modules restricted to the anti-canonical divisor}\label{subsec:support of graded modules restricted to anti-canonical divisor}
In this section, let \(A\) be a three-dimensional quadratic AS-regular algebra and let \((E, \sigma, \cL)\) be a geometric triple associated to \(A\). 
Assume that \(E\) is a smooth elliptic curve.

\begin{definition}
Let us define a subset \(E^{\mathrm{sp}}\subset E \) \footnote{sp stands for special points} as
\begin{equation}
    E^{\mathrm{sp}}
    \coloneqq
    \{
        p\in E
        \mid 
        \Supp( 
            \bL j^{ \ast }  M
        )
        =
        \{
            p
        \}
        \text{ for some }
        M\in \boundedderived \qgr( A ) 
    \}.
\end{equation}
\end{definition}

For any 
\(
    M \in \boundedderived \qgr ( A )
\), 
there exists a Leray spectral sequence
\begin{equation}
    E_{2}^{p,q} 
    =
    \bL^p j^{ \ast } H^{ q } ( M ) 
    \Rightarrow 
    \bL^{ p + q } j^{ \ast } ( M ).
\end{equation}
By \cref{lemma:derived pullback of graded modules}, we obtain that 
\(
    E_{ 2 }^{ p,q } 
    = 
    0 
\) 
if 
\(
    p
    \neq 
    0,1
\) 
and hence, this spectral sequence degenerates at \(E_{ 2 }\)-page.
Thus there exists an exact sequence 
\begin{equation*}
    0
    \to
    E_{ 2 }^{1, n - 1 } = \bL^{ 1 } j^{ \ast } H^{ n-1 }( M )
    \to
    \bL^{ n } j^{ \ast } M
    \to
    E_{ 2 }^{ 0, n } = \bL^{ 0 } j^{ \ast } H^{ n } ( M )
    \to
    0
\end{equation*}
so that we have 
\begin{equation*}
    \Supp(\bL j^{\ast} M) = \bigcup_{n\in\bZ} \Supp(\bL j^{\ast} H^{n}(M)).
\end{equation*}
This implies that 
\begin{equation}\label{eq:reduction of E circ}
    E^{\mathrm{sp}}
    =         
    \{
        p\in E
        \mid
        \Supp(\bL j^{\ast} M) = \{p\} \text{ for some } M\in \qgr( A )
    \}.
\end{equation}

\begin{lemma}\label{lemma:decomposition of E circ}
    There exists a decomposition of \( E^{\mathrm{sp}}\) as the following form:    
    \begin{equation*}
        E^{\mathrm{sp}}
        =
        E_{ 1 }
        \cup
        E_{ 2 }
    \end{equation*}
    where
    \begin{align}
        E_{ 1 }
        &
        =
        \{
            p\in E
            \mid 
            \Supp( 
                \bL j^{ \ast } M
            )
            =
            \{
                p
            \}
            \text{ for some $g$-torsion module } M
        \},\\
        E_{ 2 }
        &
        =
        \{
            p\in E
            \mid 
            \Supp( 
                \bL j^{ \ast } M
            )
            =
            \{
                p
            \}
            \text{ for some } M
            \text{ such that } \bL^1 j^{ \ast } M = 0
        \}.
    \end{align}
\end{lemma}
\begin{proof}
    For any 
    \(
        p
        \in 
        E^{\mathrm{sp}}
    \), 
    there exists 
    \(
        M \in \qgr( A )
    \) 
    such that 
    \(
        \Supp( 
            \bL j^{ \ast } M
         )
         =
         \{
            p
         \}
    \)
    by \eqref{eq:reduction of E circ}.
    Set 
    \(
        Q 
        \coloneqq 
        M / K_{ \infty }
    \)
    where \(K_{ \infty }\) is defined in \cref{def:the module K n}.
    Note that \(K_{ \infty }\) is a \(g\)-torsion module by definition and \(\bL^1 j^{ \ast } Q = 0\) by \cref{lemma:derived pullback of graded modules}.
\begin{comment}
    Indeed, considering the canonical exact triangle 
    \begin{equation*}
        \begin{tikzcd}
            Q(-3)
            \arrow[r, "\cdot g"]
            &
            Q
            \arrow[r]
            &
            j_{ \ast } \bL j^{ \ast } Q
            \arrow[r]
            &
            Q(-3)[1],
        \end{tikzcd}
    \end{equation*}
    we have 
    \(
        \bL^{ 1 } j^{ \ast } Q = 0
    \)
    by the injectivity of the multiplication map \([\cdot g]\) defined in \eqref{eq:right multiplication of g}, 
    which follows from the definition of \(Q\).
\end{comment}
    Then by the long exact sequence
    \begin{equation}
        \begin{tikzcd}
            0
            \arrow[r]
            &
            \bL^1 j^{\ast} K_{ \infty }
            \arrow[r,"\sim"]
            &
            \bL^1 j^{\ast} M 
            \arrow[r]
            &
            \bL^1 j^{\ast} Q = 0 
            \arrow[out=-10, in=170]{dll}
            &
            \\
            &
            j^{\ast} K_{ \infty }
            \arrow[r]
            &
            j^{\ast} M
            \arrow[r]
            &
            j^{ \ast } Q
            \arrow[r]
            &
            0,
        \end{tikzcd}
    \end{equation}
    we have 
    \(
        \Supp(
            \bL j^{ \ast } K_{ \infty }
        )
        \subset 
        \Supp(
            \bL j^{ \ast } M 
        )
        =
        \{
            p
        \}.
    \)
    If 
    \(
        \Supp(
            \bL j^{ \ast } K_{ \infty }
        )
        =
        \{
            p
        \}
    \),
    then \(p\in E_{ 1 }\)
    since \(K_{ \infty }\) is a \(g\)-torsion module.
    If 
    \(            
        \Supp(
            \bL j^{ \ast } K_{ \infty }
        )
        = 
        \emptyset
    \), that is 
    \(
        \bL j^{ \ast } K_{ \infty } = 0
    \),
    then 
    \(
        K_{ \infty } = 0
    \) 
    by \cref{lemma:support has non trivial intersection with elliptic curve}.
    Hence 
    \(
        M
        \simeq 
        Q
    \) 
    and  
    \(
        p \in E_{ 2 }
    \) 
    since 
    \(
        \bL^{1} j^{ \ast } Q = 0
    \).
\end{proof}

\begin{lemma}\label{lemma:vanishing of g-torsion module}
    Let 
    \(
        M\in\grmod( A )
    \). 
    Assume that \(M\) is a \(g\)-torsion module. 
    Then the followings are equivalent.
    \begin{enumerate}
        \item \( M = 0 \) as objects of \( \qgr( A ) \), i.e., \(M\) is a finite dimensional module.
        \item \( j^{ \ast } M = 0 \).
        \item  \( \bL^1 j^{ \ast } M = 0 \).
        \item \(\bL j^{ \ast } M = 0\).
    \end{enumerate}
\end{lemma}
\begin{proof}
    The implications
    \(
       (4) \Rightarrow (2)
    \)
    and
    \(
        (4) \Rightarrow (3)
    \)
    are obvious.
    Also, the implication 
    \(
        (1) \Rightarrow (4)
    \) 
    is obvious. In the rest of the proof we show the implications
    \(
       (3) \Rightarrow (1)
    \)
    and
    \(
       (2) \Rightarrow (1)
    \)
    separately.

    \(
        ( 3 )
        \Rightarrow
        ( 1 )\colon
    \)
    By \cref{lemma:derived pullback of graded modules}, we obtain 
\begin{comment}
    The canonical exact triangle
    \begin{equation*}
        M( -3 )
        \xrightarrow{\cdot g}
        M
        \to
        j_{ \ast } \bL j^{ \ast } M
        \to
        M( -3 )[ 1 ]
    \end{equation*}
    in
    \(
       \derived ^{ b } \qgr A
    \)
    yields the long exact sequence
    \begin{equation*}
        0
        \to
        j_{ \ast } \bL^{ 1 } j^{ \ast } M
        \to
        M( -3 )
        \xrightarrow{\cdot g}
        M
        \to
        j_{ \ast }j^{\ast} M
        \to 
        0
    \end{equation*}
    in
    \(
       \qgr ( A )
    \),
    which implies that 
\end{comment}
    \( 
        j_{ \ast }\bL^{1}j^{\ast} M \simeq K_{ 1 }(-3)
    \) 
    where \(K_{ n }\) is defined in \cref{def:the module K n} for \(n\in\bZ_{ >0 }\).
    Hence the assumption \(\bL^{ 1 }j^{ \ast } M = 0\) of the item (3) implies \(K_{ 1 } = 0\) in \(\qgr(A)\).
    For any \(\ell>0\), we have an exact sequence
    \begin{equation*}
        0
        \to
        K_{ \ell -1 }
        \to
        K_{ \ell }
        \xrightarrow{g^{ \ell -1 }}
        K_{ 1 }( 3\ell -3 )
    \end{equation*}
    in \(\grmod ( A )\), which hence implies
    \(
       K _{ \ell - 1 }
       \simeq
       K _{ \ell }
    \)
    in
    \(
       \qgr ( A )
    \)
    for all
    \(
       \ell > 0
    \).
    This implies the item (1), since the assumption that
    \(
       M
    \)
    is a
    \(
       g
    \)-torsion module implies
    \(
       M = K _{ \ell }
    \)
    for
    \(
       \ell \gg 0
    \).

    \(
       ( 2 )
       \Rightarrow
       ( 1 )\colon
    \)
    For any \( \ell > 0 \), there exists an exact sequence
    \begin{equation*}
        \left( 
            M/ M g^{ \ell } 
        \right) ( -3 )
        \xrightarrow{\cdot g}
        M/ M g^{ \ell + 1 }
        \to 
        M/ M g
        \to
        0
    \end{equation*}
    in
    \(
       \grmod ( A )
    \),
    where 
    \(
        Mg^{ \ell } 
        \coloneqq 
        \Image ( 
            M( -3 \ell )
            \xrightarrow{
                g ^{ \ell }
            }
            M 
        )
    \).
    Combined with the assumption
    \(
       M / M g
       =
       j ^{ \ast } M
       =
       0
       \in
       \qgr ( A )
    \),
    this inductively implies that
    \(
       M / M g ^{ \ell }
       =
       0
       \in
       \qgr ( A )
    \)
    for
    \(
       \ell > 0
    \).
    This implies the item (1), since the assumption that
    \(
       M
    \)
    is a
    \(
       g
    \)-torsion module implies
    \(
        M
        =
        M / 0
        =
        M / M g ^{ \ell }
    \)
    for
    \(
       \ell \gg 0
    \).
\end{proof}

\begin{lemma}\label{lemma:support of derived pullback of g-torsion module}
    Let \(M\in\grmod ( A )\). 
    Assume that \(M\) is annihilated by \(g\), i.e.,
    \(
        j_{ \ast }\cF
        \simeq
        M
    \)
    for some \(\cF\in\coh(E)\).
    Then 
    \begin{equation*}
        \Supp(
            \bL^{ 1 } j^{ \ast } M
        )
        =
         \sigma^{ -3 }
         \Supp(
            \cF
         ).
    \end{equation*}
\end{lemma}
\begin{proof}
    If 
    \(
        \Supp( \cF )
        =
        \emptyset
    \),
    then \(j^{ \ast } M \simeq \cF \simeq 0\). 
    The implication
    \(
       ( 2 )
       \Rightarrow
       ( 3 )
    \)
    of~\cref{lemma:vanishing of g-torsion module} implies that \(\bL^{ 1 }j^{ \ast } M \simeq 0\) and 
    hence, the assertion holds in this case.
    We may assume that \( \cF\neq 0 \).
\begin{comment}
    From the canonical exact triangle 
    \begin{equation*}
        M( -3 )
        \to
        M
        \to
        j_{ \ast } \bL j^{ \ast } M
        \to
        M( -3 )[ 1 ]
    \end{equation*}
    we have an exact sequence 
    \begin{equation*}
        0
        \to
        j_{ \ast } \bL^{ 1 } j^{ \ast } M
        \to
        M( -3 )
        \xrightarrow{\cdot g}
        M
        \to
        j_{ \ast } \bL^{ 0 } j^{ \ast } M
        \to
        0.
    \end{equation*}
    Since \( M \) is annihilated by \( g \), 
    the middle arrow of the exact sequence is trivial and 
    hence, we have 
    \(
        j_{ \ast } \bL^{ 1 } j^{ \ast } M 
        \simeq 
        M( -3 ) 
    \).
\end{comment}
    Since \(M\) is annihilated by \(g\), we have \(M = K_{ 1 }\) where \(K_{ 1 }\) is defined in \cref{def:the module K n}.
    By \cref{lemma:derived pullback of graded modules}, we have 
    \begin{equation*}
        \bL^{ 1 }j^{ \ast } M
        \simeq 
        j^{ \ast } \left( M ( -3 ) \right)
        \simeq 
        j^{ \ast } ((j_{ \ast }\cF)( -3 )).
    \end{equation*}
    By \cref{lemma: explicit desdription of grade shift on coh E}, we have 
    \begin{equation*}
        \left(
            j_{ \ast } \cF
        \right) ( -1 )
        \simeq
        j_{ \ast } (
            \sigma^{ -1 }_{ \ast }(
                \cF
            )
            \otimes 
            \cL^{ -1 }
        ).
    \end{equation*}
    This inductively implies that
    \begin{equation*}
        \left(
            j_{ \ast }\cF
        \right)
        ( -3 )
        \simeq 
        j_{ \ast }
        \left(
            \left(
                \sigma^{ -3 }_{ \ast } \cF
            \right)
            \otimes 
            \sigma^{ 2 \ast}\cL^{-1}
            \otimes 
            \sigma^{ \ast } \cL^{-1}
            \otimes 
            \cL^{-1}
        \right)
    \end{equation*}
    and hence, 
    \(
        \bL^{ 1 } j^{ \ast } M
        \simeq 
        \left(
            \sigma^{ -3 }_{ \ast } \cF
        \right)
        \otimes 
        \sigma^{ 2 \ast}\cL^{-1}
        \otimes 
        \sigma^{ \ast } \cL^{-1}
        \otimes 
        \cL^{-1}
    \).
    Therefore, we have 
    \begin{equation*}
        \Supp( 
            \bL^{ 1 } j ^{\ast} M
         )
        =
        \Supp(
            \sigma^{ -3 }_{\ast} \cF
        )
        =
        \sigma^{ -3 }
        \Supp( 
            \cF
        ).
    \end{equation*}
\end{proof}

\begin{lemma}\label{lemma:support of E_1}
    We have
    \begin{equation}
        E_{ 1 }
        \subset 
        \bigcup_{n\in \bZ_{>0}}
        E ^{ \sigma ^{ 3 n } },
    \end{equation}
    where 
    \(     
        E ^{ \sigma ^{ 3 n } }
        =
        \{
            p\in E
            \mid 
            \sigma^{ 3n }( p )
            =
            p
        \}
    \)
    is the fixed locus of the automorphism
    \(
        \sigma ^{ 3 n }
    \)
    of
    \(
       E
    \).
    In particular, if the automorphism \(\sigma^{ 3 }\) has infinite order, 
    then \(E_{ 1 } = \emptyset\).
\end{lemma}
\begin{proof}
    Take a point 
    \(
        p\in E_{ 1 }
    \).
    Then there exists a \(g\)-torsion module \(M\) such that 
    \(
        \Supp( 
            \bL j^{ \ast } M 
        ) 
        = 
        \{ p \}
    \). 
    The equivalence of the conditions
    (2) and (3) of~\cref{lemma:vanishing of g-torsion module} implies that both \(j^{ \ast }M\) and \(\bL^{1} j^{ \ast } M \) are not trivial, so that
    \begin{equation}\label{eq:support of M}
        \Supp( j^{ \ast } M  )
        =
        \Supp( \bL^{ 1 } j^{ \ast } M )
        =
        \{p\}.
    \end{equation}
    Since \( M \) is a \(g\)-torsion module, \( M = K_{ n } \) for some \(n\), where \( K_{ n } \) is defined in~\cref{def:the module K n}.
    We have a filtration 
    \begin{equation*}
        0 = K_{0}
        \subset 
        K_1 
        \subset 
        \cdots
        \subset 
        K_{n-1}
        \subset
        K_n = M
    \end{equation*}
    whose successive quotients 
    \(
        K_{ \ell } / K_{ \ell - 1 }
    \) 
    come from 
    \(
        \coh(E)
    \), i.e., there are coherent sheaves \( \cF_{ \ell } \) on \(E\) such that 
    \begin{equation}\label{eq:succ quot of filtration is isom to coh on E }
        j_{\ast} \cF_{ \ell } 
        \simeq 
        K_{ \ell } / K_{ \ell - 1 }
    \end{equation}
    in 
    \(
        \qgr( A )
    \)
    since 
    \(
        K_{ \ell } / K_{ \ell-1 }
    \)
    are 
    \(
        A / g
    \)-modules for any \( \ell \).
    For any \(\ell\geq 1\), we have a long exact sequence
    \begin{equation}\label{eq:long exact sequence 3}
        \begin{tikzcd}
            0
            \arrow[r]
            &
            \bL^1 j^{\ast} K_{\ell-1}
            \arrow[r, "\alpha^1_{ \ell }"]
            &
            \bL^1 j^{\ast} K_{\ell} 
            \arrow[r, "\beta^1_{ \ell }"]
            &
            \bL^1 j^{\ast} j_{\ast} \cF_{ \ell }
            \arrow[out=-10, in=170, "\delta_{ \ell }"']{dll}
            &
            \\
            &
            j^{\ast} K_{\ell-1}
            \arrow[r, "\alpha^0_{ \ell }"]
            &
            j^{\ast} K_{ \ell }
            \arrow[r, "\beta^0_{ \ell }"]
            &
            \cF_{ \ell }
            \arrow[r]
            &
            0.
        \end{tikzcd}
    \end{equation}
    Applying 
    \eqref{eq: exact triangle associaed to adjunction}
    to 
    \(K_{ \ell }\hookrightarrow M\) with \(F= \bL j^{ \ast }\) and \(R=j_{ \ast }\), we obtain a morphism of exact triangles. From this exact triangle, we obtain a morphism between the following long exact sequences.
    \begin{equation*}
        \begin{tikzcd}
            0
            \arrow[r]
            &
            j_{ \ast }\bL^1 j^{\ast} M
            \arrow[r]
            &
            M( -3 )
            \arrow[r, "g"]
            &
            M
            \arrow[r]
            &
            j_{ \ast } j^{ \ast } M
            \arrow[r]
            &
            0
            \\
            0
            \arrow[r]
            &
            j_{ \ast } \bL^{ 1} j^{\ast} K_{\ell}
            \arrow[r]
            \arrow[u]
            &
            K_{ \ell }( -3 )
            \arrow[r, "g"]
            \arrow[u]
            &
            K_{ \ell }
            \arrow[r]
            \arrow[u]
            &
            j_{ \ast }j^{ \ast } K_{ \ell }
            \arrow[r]
            \arrow[u]
            &
            0.
        \end{tikzcd}
    \end{equation*}
    Then we have the following commutative diagram:
    \begin{equation*}
        \begin{tikzcd}
            \bL^{ 1 } j^{ \ast } M 
            \arrow[r, "\sim"]
            &
            j^{ \ast } 
            \left(
                K_{ M, 1 }( -3 )
            \right)
            \\
            \bL^{ 1 } j^{ \ast } K_{ \ell }
            \arrow[r,"\sim"]
            \arrow[u]
            &
            j^{ \ast } 
            \left(
                K_{ K_{ \ell }, 1 }( -3 )
            \right)
            \arrow[u]
        \end{tikzcd}
    \end{equation*}
    The right vertical arrow is an isomorphism since \(K_{ M, 1 }= K_{ K_{ \ell }, 1 }\) as graded \(A\)-module if \(\ell \geq 1\).
    Hence the left vertical arrow is an isomorphism and we have 
    \begin{equation}\label{eq: support of K ell}
        \Supp(
            \bL^{ 1 } j^{ \ast } K_{ \ell }
        )
        =
        \Supp(
            \bL^{ 1 } j^{ \ast } M
        )=
        \{p\}.
    \end{equation}
    Then we claim that:
    \begin{claim}\label{claim}
        For \(\ell \geq 1\), 
        \( \Supp( j^ {\ast} K_{ \ell } ) \subset \{p, \dots, \sigma^{ -3( n-\ell ) }(p)\} \). 
    \end{claim}
    \begin{proof}
        We prove the claim by using descending induction on \(\ell\).
        If \(\ell = n\), then the assertion holds by \eqref{eq:support of M}.
        For \(n - 1 \geq \ell\geq 1\), assume the induction hypothesis:
        \begin{equation*}
            \Supp( j^ {\ast} K_{ \ell + 1} ) 
            \subset 
            \{p, \dots, \sigma^{ -3( n-\ell - 1 ) }(p)\}
        \end{equation*}
        By \eqref{eq:long exact sequence 3}, we have 
        \begin{equation}\label{eq: support of F ell}
            \Supp( 
                \cF_{ \ell + 1} 
            )
            \subset
            \Supp( 
                j^{ \ast } K_{ \ell + 1 }
             )
            \subset 
            \{ 
                p, \dots, \sigma^{ -3 (n -\ell - 1 ) }( p )
            \}.
        \end{equation}
        Then
        \begin{equation}\label{eq: support of derived pullback of F ell}
            \Supp(
                \bL^{1} j^{ \ast } j_{ \ast } \cF_{ \ell + 1 }
            )
            \overset{\text{\cref{lemma:support of derived pullback of g-torsion module}}}{=}
            \sigma^{ -3 }
            \Supp( 
                \cF_{ \ell + 1 }
             )
             \overset{\eqref{eq: support of F ell}}{\subset}
             \{
                \sigma^{ -3 }( p )
                ,\dots,
                \sigma^{ -3( n- \ell )} (p)
             \}.
        \end{equation}
        Therefore
        \begin{equation*}
            \Supp(
                j^{\ast } K_{ \ell }
            )
            \overset{\eqref{eq:long exact sequence 3}}{\subset}
            \Supp(
                \bL^{ 1 } j^{ \ast } j_{ \ast } \cF_{ \ell + 1 }
            )
            \cup
            \Supp(
                j^{\ast} K_{ \ell + 1}
            )
            \overset{\eqref{eq: support of F ell}, \eqref{eq: support of derived pullback of F ell}}{\subset}
            \{
                p, \dots, \sigma^{ -3( n -\ell) }
            \}.
        \end{equation*}
        By induction on \(\ell\), the assertion holds.
    \end{proof}
    Therefore we have
    \begin{equation*}
        \{p\}
        \overset{\eqref{eq: support of K ell}}{=}
        \Supp(
            \bL^{ 1 }j^{ \ast } K_{ 1 }
        )
        \overset{\text{\cref{lemma:support of derived pullback of g-torsion module}}}{=}
        \sigma^{ -3 }
        \Supp(
            j^{ \ast } K_{ 1 }
        )
        \overset{\text{\cref{claim}}}{\subset}
        \{
            \sigma^{ -3 }( p ), \dots, \sigma^{ -3 n }( p )
        \}
    \end{equation*}
    and hence, \(p = \sigma^{ 3\ell }( p ) \) for some \(\ell\geq 1\).

    Finally, if \(\sigma^{ 3 }\) is not a torsion element, then \(\sigma^{ 3\ell }\) has no fixed point for any \(\ell\), i.e.,
    \(
        E ^{ \sigma ^{ 3 \ell } } = \emptyset
    \) 
    for any \(\ell\)  (see \cite[{p321}]{Hartshorne}).
    This implies 
    \(
        \bigcup_{ \ell \in\bZ_{>0} }
        E ^{ \sigma ^{ 3 \ell } } = \emptyset
    \)
    and hence, \(E_{ 1 }= \emptyset\).
\end{proof}

Let us define the map 
\begin{equation}\label{eq:degree map}
    \deg
    \colon
    \grothendieckgroup( E )
    \coloneqq 
    \grothendieckgroup( \dbcoh( E ) )
    \to
    \bZ
\end{equation}
which is given by \(\cF\mapsto \chi(\cF)\) where 
\(
    \cF\in\coh( E )
\) 
and
\(
    \chi( F ) 
    = 
    \sum (-1)^i \dim H^{ i }( E, \cF )
\).
It is known that there exists an isomorphism
\begin{equation}\label{eq: isomorphism between Chow group and Grothendieck group} 
    \picardgroup( E )
    \oplus 
    \bZ
    \simeq 
    \grothendieckgroup( E )
\end{equation}
which is defined by the direct sum of 
\(
    \sum n_{ p } p 
    \mapsto 
    \sum n[ \cO_{ p } ]
\) 
for a Weil divisor \(\sum n_{p} p\) on \(E\) and 
\(
    n\mapsto n[ \cO_{ E } ]
\)
for \(n\in\bZ\) (see \cite[{Exercise~I.6.11}]{Hartshorne}).
Note that the map defined in \eqref{eq:degree map} coincides with the map of the degree of a Weil divisor on 
\(
    \picardgroup (E)
\), 
via the map \eqref{eq: isomorphism between Chow group and Grothendieck group}.

For \(\eta\in\grothendieckgroup( E )\), let us define 
\(
    E_{ 2, \eta }\subset E_{ 2 }
\) 
as 
\begin{equation*}
    E_{ 2, \eta }
    \coloneqq
    \{
        p
        \mid
        \Supp( \bL j^{ \ast } M )
        =
        \{ p \}
        \text{ for some }
        M\in\qgr(A)
        \text{ such that }
        \bL^1 j^{ \ast } M = 0,
        [M] = \eta
    \}.
\end{equation*}
If 
\(
    \deg\bL j^{ \ast } \eta \leq 0
\), 
then 
\(
    E_{ 2, \eta } 
    =
    \emptyset
\). 
Indeed, assume 
\(
    E_{ 2, \eta }
    \neq
    \emptyset
\)
and take
\(
    p\in E_{ 2, \eta }
\).
Then there exists \( M \in \grmod( A ) \) such that 
\(
    \Supp(
        \bL j^{ \ast }M 
    )
    =
    \{p\}
\) 
and
\(
    \bL^{ 1 } j^{ \ast } M 
    =
    0
\).
 Since \( j^{\ast}M \) is supported at \(p\), we have 
 \(
    \deg \bL j^{ \ast }\eta 
    =
    \deg [j^{ \ast } M ] 
    =
    \dim H^{ 0 } ( E, j^{\ast} M) > 0.
\)

\begin{lemma}\label{lemma:support of E_2}
    For any 
    \(
        \eta\in\grothendieckgroup( A )
    \) 
    which satisfies 
    \(
        n
        \coloneqq 
        \deg \bL j^{ \ast } \eta >  0
    \), 
    \(
        E_{ 2, \eta }
    \) 
    is finite.
    In particular, 
    \(
        E_{ 2 } = \bigcup_{ \eta \in \grothendieckgroup( A )} E_{ 2,\eta }
    \) 
    is at most countable.
\end{lemma}
\begin{proof}
    Let \(p,q\in E_{2, \eta}\). Then there exist \(M, N\in \grmod( A )\) which satisfy the following conditions:
    \begin{equation*}
        \bL^1 j^{ \ast } M
        = 
        \bL^1 j^{ \ast } N
        = 
        0
        \text{ and }
        \Supp ( \bL j^{ \ast } M ) = \{ p \},\ 
        \Supp ( \bL j^{ \ast } N ) = \{ q \}.
    \end{equation*} 
    Since \( j^{ \ast } M \) and \( j^{ \ast } N \) are supported at \(p\) and \(q\) respectively, 
    we have the following equation of \(K\)-classes on \(E\)
    \begin{equation*}
            [\bL j^{ \ast } M]
            = 
            [j^{ \ast } M]
            =
            n[\cO_{p}],\ 
            [\bL j^{ \ast } N]
            =
            [j^{ \ast } N]
            =
            n[\cO_{q}]
    \end{equation*} 
    where 
    \(
        n 
        = 
        \deg \bL j^{ \ast } \eta
        > 
        0
    \).
    Since 
    \(
        \eta 
        = 
        [ M ] 
        = 
        [ N ]
    \), 
    then 
    \(
        \bL j^{ \ast }[ M ]  
        = 
        \bL j^{ \ast } [ N ]
    \). 
    i.e., 
    \begin{equation*}
        n
        \left(
            [ \cO_{ p } ]
            -
            [ \cO_{ q } ]
        \right)
        =
        0.
    \end{equation*}
    By \eqref{eq: isomorphism between Chow group and Grothendieck group}, 
    the divisor \( p - q\in \picardgroup( E ) \) is annihilated by \(n\).
    Therefore, we can define the map
    \begin{equation*}
        E_{ 2, \eta }
        \to
        E[ n ],\ p\mapsto p - q 
    \end{equation*}
    for a fixed point 
    \(
        q \in E_{ 2, \eta }
    \) 
    and this map is injective.
    Hence,
    \(
        E_{ 2, \eta } 
    \) 
    is a finite set
    since \( E [ n ] \) is a finite set. 
    Finally, since
    \(
        \grothendieckgroup( A )\simeq \bZ^{ 3 }
    \)
    by \eqref{eq:grothendieckgroup of A}, 
    the set 
    \(
        E_{ 2 } 
        = 
        \bigcup_{ \eta\in \grothendieckgroup( A ) }
        E_{ 2, \eta }
    \) 
    is at most countable.
\end{proof}

\begin{proposition}\label{proposition:support of E circ}
    If the automorphism \( \sigma \) has infinite order, 
    then the set 
    \( 
        E^{\mathrm{sp}}
    \)
    is at most countable. As a consequence, we have 
    \(
        E
        \setminus 
        E^{\mathrm{sp}}
        \neq 
        \emptyset
    \).
\end{proposition}
\begin{proof}
    It follows from \cref{lemma:decomposition of E circ}, \cref{lemma:support of E_1}, and \cref{lemma:support of E_2}.
\end{proof}

\subsection{Proof of \cref{Main result}}\label{subsec:proof of main result}
In this section, we prove the main theorem using a noncommutative analogue of \cite[Theorem~3.1]{borisov2024nonexistencephantomsnongenericblowups}, 
which addresses the commutative case.

Let \( A \) be a three-dimensional quadratic AS-regular algebra and let 
\(
    ( E, \sigma, \cL)
\) 
be a geometric triple associated to \( A \) described in \cref{theorem: AS quads are isom to geometric algebra}.
We assume that 
\( 
    E 
\) 
is an elliptic curve, and the automorphism 
\(
    \sigma
\) 
is a translation with infinite order.

\begin{theorem}\label{Main theorem}
    Let 
    \(
        \boundedderived \qgr(A)
        =
        \langle
            \cA,\cB
        \rangle
    \)
    be a semi-orthogonal decomposition with \(K_0(\cB)=0\). Then \(\cB\) is trivial.
    In particular, there are no phantom categories on noncommutative projective planes with infinite order translation.
\end{theorem}
\begin{proof}
    By \cref{proposition:skyscraper sheaves on E is a spanning class}, it is enough to show that 
    \(
        j_{ \ast } \cO_{ p }\in \cA
    \) 
    for any \( p \in E\).
    Consider the decomposition of \( j_{ \ast } \cO_{ p }\)
    \begin{equation}
        B_{ p }
        \to
        j_{ \ast } \cO_{ p }
        \to
        A_{ p }
        \to
        B_{ p }[1]
    \end{equation}
    where \(A_{ p }\in \cA\) and \(B_{ p }\in \cB\).
    Then we only have to prove that \(\bL j^{ \ast } B_{ p } = 0\) for any \(p\in E\) 
    since \(B_{ p } = 0\) if and only if \(\bL j^{ \ast } B_{ p } = 0\) by \cref{lemma:support has non trivial intersection with elliptic curve}.
    
    By \cref{corollary: an exact triangle associated to the derived pullback which is a spherical functor},
    there exists an exact triangle on \(E\)
    \begin{equation}\label{eq:exact tri ass to sphericl functor}
        \bL j^{ \ast } B_{ p }
        \to
        \cO_{ p }
        \to
        C_{ p } 
        \coloneqq
        T( \cO_{ p } )
        \to
        \bL j^ { \ast } B_{ p }[ 1 ]
    \end{equation}
    where \(T\) is a spherical twist associated to 
    \(
        \cB
        \hookrightarrow 
        \derived^ {b}\qgr ( A ) 
        \overset{\bL j^{ \ast }}{\to} 
        \dbcoh ( E ) 
    \).
    Since \(\cO_{ p }\) is a spherical object, then so is \( C_{ p } \). 
    Moreover, since the \(K\)-class of 
    \(
        \bL j^{ \ast }B_{ p }
    \) 
    vanishes by our assumption for \(\cB\), we have 
    \begin{equation}\label{eq: Cp eqaulas to Op as K-class}
        [ \cO_{ p } ] 
        = 
        [ C_{ p } ]
    \end{equation}
    in \( \grothendieckgroup( E ) \).
    It is known that any spherical object on \(E\) is isomorphic to either a simple vector bundle or a skyscraper sheaf up to shift by \cite[Proposition~4.13]{BK2006}.
    Hence the equation \eqref{eq: Cp eqaulas to Op as K-class} implies that for any \(p\in E\), there exists an isomorphism 
    \begin{equation}\label{eq:of the form of Cp}
        C_ { p }
        \simeq 
        \cO_{ p }[ 2 a_{ p } ]
    \end{equation}
    for some \( a_{ p }\in \bZ \).
    Since \(T\) sends skyscraper sheaves to skyscraper sheaves up to shift, the autoequivalence \(T\) is of the form 
    \begin{equation}
        T
        \simeq 
        \rho_{ \ast }
        ( - \otimes \cL )[ n ]
    \end{equation}
    for some \(\rho\in\Aut ( E )\), \(\cL\in\picardgroup( E )\), and \(n\in \bZ\) by \cite[Corollary~4.3]{hille2005fouriermukaitransforms}.
    Then we claim that \( \rho = \id_{E}\), \(n = 0\), and \(\cL\simeq \cO_{ E }\).
    Indeed, by \eqref{eq:of the form of Cp}, we have \(\rho = \id_{E}\). 
    By \eqref{eq:exact tri ass to sphericl functor}, we have 
    \(
        \Supp ( \bL j^{ \ast }B_{ p } ) 
        \subset 
        \{ p\}
    \).
    By \cref{proposition:support of E circ}, there exists a point \(p\notin E^{\mathrm{sp}}\).
    If 
    \(
        p 
        \notin 
        E^{\mathrm{sp}}
    \), 
    then 
    \(
        \bL j^{ \ast } B_{ p } = 0
    \). 
    Moreover, by \eqref{eq:exact tri ass to sphericl functor} we have 
    \begin{equation}\label{eq:isom on E circ}
        \cO_{ p }
        \simeq 
        C_{ p }
    \end{equation}
    for 
    \(
        p
        \notin 
        E^{\mathrm{sp}}
    \).
    This implies \(n = 0\). 
    Finally, consider the exact triangle obtained by \eqref{eq:exact triangle associated to spherical functor}
    \begin{equation*}
        \bL j^{ \ast } B_{ \cO }
        \to
        \cO_{ E }
        \to
        \cL
        \to
        \bL j^{ \ast } B_{ \cO }[ 1 ]
    \end{equation*}
    where 
    \( 
        B_{ \cO } 
        = 
        \pr^{ R }_{ \cB } ( j_{ \ast }\cO_{ E } )
    \).
    The equation of the \( K \)-classes
    \begin{equation*}
        [ \cO_{ E } ]
        =
        [ \cL ]
    \end{equation*}
    implies that 
    \(
        \cO_{ E }\simeq \cL
    \).

    Therefore, the morphism of kernels on 
    \( 
        E \times E
    \) 
    that corresponds to 
    \(
        \id
        \to
        T
    \)
    is given by some element 
    \(
        \psi
    \)
    of 
    \(
        \Hom_{ E\times E } ( \cO_{ \Delta }, \cO_{ \Delta } )
    \).
    Since the morphism space is one-dimensional, \(\psi\) is either an isomorphism or a trivial morphism.
    By \eqref{eq:isom on E circ}, \( \psi \) is not trivial and hence it is an isomorphism.
    As a consequence, we have \(\bL j^{ \ast } B_{ p } = 0\) for any \(p\in E\) and the assertion holds.
\end{proof}
%
%
\appendix
\section{Proof of \cref{corollary: an exact triangle associated to the derived pullback which is a spherical functor}}\label{section: proof of corollary}

In this section we prove \cref{corollary: an exact triangle associated to the derived pullback which is a spherical functor}.
Let \(Q\) be a finite acyclic quiver and \(\Lambda = \bfk Q / I\), where \(I\) is a two-sided ideal of the path algebra \(\bfk Q\).

Since \(Q\) is finite and acyclic, the algebra \(\Lambda\) has finite global dimension.
Moreover, the enveloping algebra \(\Lambda^e= \Lambda^{\op}\otimes_{\bfk}\Lambda\) has finite global dimension. Indeed, \(\Lambda^e\) is a quotient algebra of the path algebra \(\bfk(Q^{\op}\times Q)\), where \(Q^{\op}\) is the opposite quiver of \(Q\) and \(Q^{\op}\times Q\) is the product of the quivers (see \cite[{Proposition~3}]{MR2357345}). 
As \(Q^{\op}\times Q\) is finite and acyclic, the enveloping algebra \(\Lambda^{e}\) has finite global dimension as well.

\begin{lemma}\label{lemma: adjunction associaed to bimodule}
    Let \(\Lambda\) be a finite dimensional algebra and let \(\Lambda^{ e }=\Lambda^{\op}\otimes_{\bfk}\Lambda\) be an enveloping algebra. For a right \(\Lambda\)-module \(W\), there exists the following adjunction
    \begin{equation*}
        _{\Lambda}(-)\otimes_{\bfk} W_{ \Lambda }
        \colon 
        \module \Lambda^{\op}
        \rightleftarrows
        \module \Lambda^e
        \colon 
        _{ \Lambda }\Hom_{\module \Lambda}( W_{ \Lambda }, (-) ),
    \end{equation*}
    where a left module structure of \(_{ \Lambda }\Hom_{\module \Lambda}( W_{ \Lambda }, M)\) for a bimodule \(M\) is inherited from the left module structure of \(M\).
    In particular, we have an adjunction
    \begin{equation*}
        _{\Lambda}(-)\Lotimes_{\bfk} W_{ \Lambda } 
        =
        _{\Lambda}(-)\otimes_{\bfk} W_{ \Lambda }
        \colon 
        \boundedderived\module \Lambda^{\op}
        \rightleftarrows
        \boundedderived\module \Lambda^e
        \colon 
        _{ \Lambda }\RHom_{\module \Lambda}( W_{ \Lambda }, (-) ).
    \end{equation*}
\end{lemma}

\begin{lemma}\label{lemma: semi-orthogonal decomposition of derived category of a enveloping algebra}
    Let
    \( 
        \boundedderived \module \Lambda
        =
        \langle 
            \cA, \cB
        \rangle
    \) 
    be a semi-orthogonal decomposition.
    Then there exists a semi-orthogonal decomposition 
    \begin{equation*}
        \boundedderived\module \Lambda^{ e }
        =
        \langle 
        \boundedderived\module \Lambda^{ \op } \otimes_{ \bfk } \cA, 
        \boundedderived\module \Lambda^{ \op } \otimes_{ \bfk } \cB 
        \rangle
    \end{equation*} 
    where 
    \( 
        \Lambda^{ e } = \Lambda^{ \op } \otimes_{ \bfk } \Lambda
    \)
    is the enveloping algebra.
\end{lemma}
\begin{proof}
    Since \(\Lambda^{e}\) has finite global dimension, the derived category \(\boundedderived \module \Lambda^e\) is the smallest triangulated category containing projective modules over \(\Lambda^e\).
    Since \(\Lambda^e\) is the quotient of the path algebra \(\bfk (Q^{\op}\times Q)\), any indecomposable projective module is of the form
    \begin{equation*}
        P_{( i,j )}
        = 
        \Lambda e_{ j }\otimes_{ \bfk } e_{ i }\Lambda
        \in
        \module \Lambda^{ \op } \otimes_{\bfk} \module \Lambda
    \end{equation*}
    where \(i,j\) are vertices of \(Q\). 
    Therefore, we have 
    \begin{equation*}
        \boundedderived \module \Lambda^e 
        = 
        \langle
            \module \Lambda^{ \op } \otimes_{\bfk} \module \Lambda
        \rangle
    \end{equation*}
    and this shows that 
    \(  
        \boundedderived\module \Lambda^{ \op } \otimes_{ \bfk } \cA
    \)
    and 
    \(
        \boundedderived\module \Lambda^{ \op } \otimes_{ \bfk } \cB 
    \)
    generate \(\boundedderived \module \Lambda^{ e }\) as a triangulated category.
    By the adjunction \cref{lemma: adjunction associaed to bimodule}, for \(M, N\in\module\Lambda^{\op}\), \(A\in\cA\) and \(B\in\cB\), we have 
    \begin{align*}
        \RHom_{ \Lambda^e }( M\otimes_{ \bfk } B, N\otimes_{\bfk} A)
        &
        =
        \RHom_{\module \Lambda^{\op} }( M, \RHom_{\module\Lambda}(B, N\otimes_{ \bfk } A) )
        \\
        &=
        \RHom_{\module \Lambda^{\op} }( M, N\otimes _{ \bfk }\RHom_{\module\Lambda}(B, A) )
        \\
        &=
        \RHom_{\module \Lambda^{\op} }( M, N ) \otimes _{ \bfk }\RHom_{\module\Lambda}(B,A)
        \\
        &= 0.
    \end{align*}
    This implies the semi-orthogonality.
\end{proof}

\begin{lemma}\label{lemma: admissible subcategory of finite dimensional algebra has dg enhancement}
    Let 
    \( 
        \cB 
    \) 
    be an admissible subcategory of the 
    \( 
        \boundedderived \module \Lambda 
    \).
    Then the right and the left projections onto
    \( 
        \cB 
    \) 
    have dg enhancements.
\end{lemma}
\begin{proof}
For the diagonal \( \Lambda \)-bimodule \(\Lambda\),
there exists an exact triangle
\begin{equation*}
    P_{ \cB }
    \to
    \Lambda
    \to
    P_{ \cA }
    \to
    P_{ \cB }[ 1 ],
\end{equation*}
where \( P_{ \cA } \) and \( P_{ \cB } \) are objects belonging to the subcategories
\(
    \boundedderived\module \Lambda^{ \op } \otimes_{ \bfk } \cA
\)
and
\(
    \boundedderived\module \Lambda^{ \op } \otimes_{ \bfk } \cB
\), respectively.
The functors 
\(
    \Phi_{ P_{ \cA } }
    \colon
    M
    \mapsto
    M 
    \Lotimes_{ \Lambda }
    P_{ \cA }
\)
and 
\(
    \Phi_{ P_{ \cB } }
    \colon
    M
    \mapsto
    M
    \Lotimes_{ \Lambda }
    P_{ \cB }
\)
correspond to the projection functors onto the admissible subcategories 
\(
    \cA
\)
and
\(
    \cB
\),
respectively.
Hence the projection functors onto 
\(
    \cA
\)
and
\(
    \cB
\)
have dg enhancements.
In particular, the right projection onto \(\cB\) has a dg enhancement.
By applying the previous discussion to the semi-orthogonal decomposition
\(
\langle
\cB, \cB^{\perp}
\rangle
\),
the left projection onto (\cB) possesses a dg enhancement.
\end{proof}

\begin{proposition}[{= \cref{corollary: an exact triangle associated to the derived pullback which is a spherical functor}}]    
    Let 
    \( 
        \cB 
        \subset 
        \boundedderived \qgr(A) 
    \) 
    be an admissible subcategory and
    \(
        \pr^{ R }_{ \cB }
    \) 
    be the right projection onto \( \cB \).
    There exists an exact triangle in 
    \( 
        \dbcoh( E \times E ) 
    \)
    \begin{equation*}
        K_{ LR }
        \to
        \cO_{ \Delta }
        \to
        K_{ T }
        \to
        K_{ LR }[ 1 ]
    \end{equation*}
    such that 
    \(
        \Phi_{ K_{ LR } }
        \simeq 
        \bL j^{ \ast } \pr^{ R }_{ \cB } j_{ \ast }
    \)
    and 
    \(
        \Phi_{ K_{ T } }
        \simeq 
        T
    \),
    where \( T \) is a spherical twist functor associated to the spherical functor 
    \(
        \cB
        \hookrightarrow 
        \boundedderived \qgr( A )
        \to 
        \dbcoh(E)
    \).
    Thus for any \( F \in \dbcoh( E ) \), there exists an exact triangle
    \begin{equation*}
        \bL j^{ \ast } \pr^{ R }_{ \cB } (j_{ \ast } F)
        \to
        F
        \to
        T( F )
        \to
        \bL j^{ \ast } \pr^{ R }_{ \cB } ( j_{ \ast } F )[ 1 ]
    \end{equation*}
    in \( \dbcoh( E ) \).
\end{proposition}
\begin{proof}
Since 
\(
    \boundedderived \qgr(A)
\) 
admits a full strong exceptional collection, 
there is a tilting object \( G \) by \cref{theorem: full strong exceptional collection on fake p2}.
We apply the lemmas above to \(\Lambda=\End(G)\).
There is a dg adjunction 
\begin{align}\label{eq: Hom tensor adjoint}
    -\otimes_{\Lambda} G
    \colon
    \dgPerf(\Lambda)
    \rightleftarrows
    \dgPerf( A )
    \colon
    \Hom( G, - )
\end{align}
where 
\(
    \dgPerf( \Lambda )
\)
is a dg enhancement of the triangulated category of perfect dg modules  
\(
    \Perf( \Lambda )
\)
and
\(
    \dgPerf( A )
\)
is a dg enhancement of 
\(
    \boundedderived\qgr ( A )
\) (see \cite{MR1258406}).
Since \( G \) is a tilting object, 
the adjunction induces an equivalence of triangulated categories between its homotopy categories.
\begin{align}\label{eq:Hom-tensor equivalence by tilting object}
    -\Lotimes G
    \colon
    \boundedderived \module \Lambda
    \simeq 
    \boundedderived \qgr( A )
    \colon
    \RHom ( G, - )
\end{align}

Let \(\cB_{\Lambda}\) be a full dg subcategory of 
\(
    \dgPerf(\Lambda)
\) 
which corresponds to 
\(
    \cB
    \subset 
    \boundedderived \qgr( A )
\) 
via the equivalence \eqref{eq:Hom-tensor equivalence by tilting object}.
By \cref{lemma: admissible subcategory of finite dimensional algebra has dg enhancement}, 
there exists a dg adjunction.
\begin{equation*}
    \begin{tikzcd}
        \cB_{ \Lambda }
        \arrow[r, shift left=1, hook]
        &
        \dgPerf ( \Lambda )
        \arrow[l, shift left=1]
    \end{tikzcd}
\end{equation*}
We have a dg adjunction 
\(
    L
    \dashv 
    R
\)
\begin{equation*}
    \begin{tikzcd}
        L
        \colon
        \cB_{\Lambda}
        \arrow[r, shift left=1, hook]
        &
        \dgPerf(\Lambda)
        \arrow[r, shift left = 1, "\eqref{eq: Hom tensor adjoint}"]
        \arrow[l, shift left = 1]
        &
        \dgPerf(A)
        \arrow[r, shift left = 1, "\bL j^{ \ast }"]
        \arrow[l, shift left = 1]
        &
        \dgPerf(E)
        \colon 
        R
        \arrow[l, shift left = 1, "j_{ \ast }"],
    \end{tikzcd}
\end{equation*}
where 
\(
    \dgPerf( E )
\)
is a dg enhancement of the triangulated category of perfect complexes on 
\( 
    E 
\).
From the adjunction 
\( 
    L \dashv R 
\), 
we have an exact triangle of dg endofunctors of 
\(
    \dgPerf( E )
\)
\begin{equation}\label{eq: exact triangle associated to adjoint}
    LR
    \xrightarrow{\epsilon}
    \id
    \to
    T\coloneqq \Cone(\epsilon)
    \to
    LR[ 1 ]
\end{equation}
where 
\(
    \epsilon
    \colon 
    LR
    \to 
    \id_{ \dgPerf( E ) }
\) 
is a counit.
By \cite[{Theorem~8.9}]{Toen2007}, there exists an exact triangle
\begin{equation}\label{eq: exact triangle of kernels}
    K_{ LR }
    \to
    \cO_{ \Delta }
    \to
    K_{ T }
    \to
    K_{ LR }[ 1 ]
\end{equation}
in 
\(
    \dbcoh( E \times E )
\)
such that the exact triangle of integral functors associated to \eqref{eq: exact triangle of kernels}
corresponds to \eqref{eq: exact triangle associated to adjoint}.
By construction, 
\(
    H^{ 0 } ( \Phi_{ K_{ LR } } )
    \simeq 
    \bL j^{ \ast } \pr^{ R }_{ \cB } j_{ \ast }
\)
as the exact endofunctors of 
\( 
    \dbcoh( E )
\).
We have an exact triangle in \( \dbcoh ( E ) \)
\begin{equation*}
    \bL j^{ \ast } \pr^{ R }_{ \cB }j_{ \ast }( F )
    \to
    F
    \to
    T( F )
    \to
    \bL j^{ \ast } \pr^{ R }_{ \cB }j_{ \ast } ( F )[ 1 ]
\end{equation*}
for any object \( F \in \dbcoh( E ) \).
\end{proof}

\printbibliography

@article{MR3062745,
	author = {B{{\"o}}hning, Christian and von Bothmer, Hans-Christian Graf and Sosna, Pawel},
	doi = {10.1016/j.aim.2013.04.017},
	fjournal = {Advances in Mathematics},
	issn = {0001-8708},
	journal = {Adv. Math.},
	mrclass = {14F05 (14J29)},
	mrnumber = {3062745},
	mrreviewer = {Caryn Werner},
	pages = {203--231},
	title = {On the derived category of the classical {G}odeaux surface},
	url = {http://dx.doi.org/10.1016/j.aim.2013.04.017},
	volume = {243},
	year = {2013}}

@article{MR3090263,
	author = {Gorchinskiy, Sergey and Orlov, Dmitri},
	doi = {10.1007/s10240-013-0050-5},
	fjournal = {Publications Math\'{e}matiques. Institut de Hautes \'{E}tudes Scientifiques},
	issn = {0073-8301},
	journal = {Publ. Math. Inst. Hautes \'{E}tudes Sci.},
	mrclass = {14F05 (14C15 18E30)},
	mrnumber = {3090263},
	mrreviewer = {Andrei D. Halanay},
	pages = {329--349},
	title = {Geometric phantom categories},
	url = {https://mathscinet.ams.org/mathscinet-getitem?mr=3090263},
	volume = {117},
	year = {2013}}

@article{MR1846352,
	author = {Van den Bergh, Michel},
	doi = {10.1090/memo/0734},
	fjournal = {Memoirs of the American Mathematical Society},
	issn = {0065-9266},
	journal = {Mem. Amer. Math. Soc.},
	mrclass = {16S38 (14A22 16E99)},
	mrnumber = {1846352},
	mrreviewer = {Colin J. Ingalls},
	number = {734},
	pages = {x+140},
	title = {Blowing up of non-commutative smooth surfaces},
	url = {https://mathscinet.ams.org/mathscinet-getitem?mr=1846352},
	volume = {154},
	year = {2001}}

@article {ATV2,
   AUTHOR = {Artin, M. and Tate, J. and Van den Bergh, M.},
	TITLE = {Modules over regular algebras of dimension {$3$}},
  JOURNAL = {Invent. Math.},
 FJOURNAL = {Inventiones Mathematicae},
   VOLUME = {106},
	 YEAR = {1991},
   NUMBER = {2},
	PAGES = {335--388},
	 ISSN = {0020-9910},
  MRCLASS = {16W50 (14H52 16E10)},
 MRNUMBER = {1128218},
MRREVIEWER = {G\"{u}nter R. Krause},
	  DOI = {10.1007/BF01243916},
	  URL = {https://doi.org/10.1007/BF01243916},}

@article {MR2836401,
	  AUTHOR = {Van den Bergh, Michel},
	   TITLE = {Noncommutative quadrics},
	 JOURNAL = {Int. Math. Res. Not. IMRN},
	FJOURNAL = {International Mathematics Research Notices. IMRN},
		YEAR = {2011},
	  NUMBER = {17},
	   PAGES = {3983--4026},
		ISSN = {1073-7928},
	 MRCLASS = {14A22 (16S38)},
	MRNUMBER = {2836401},
  MRREVIEWER = {Adam Nyman},
		 DOI = {10.1093/imrn/rnq234},
		 URL = {https://doi.org/10.1093/imrn/rnq234},}

@article{BK2006,
		title={Derived categories of irreducible projective curves of arithmetic genus one},
		author={Burban, Igor and Kreu{\ss}ler, Bernd},
		year={2006},
		volume={142},
		DOI={10.1112/S0010437X06002090},
		number={5},
		journal={Compositio Mathematica},
		pages={1231–1262},}

@article{PIROZHKOV2023109046,
		title = {Admissible subcategories of del Pezzo surfaces},
		journal = {Advances in Mathematics},
		volume = {424},
		pages = {109046},
		year = {2023},
		issn = {0001-8708},
		doi = {https://doi.org/10.1016/j.aim.2023.109046},
		url = {https://www.sciencedirect.com/science/article/pii/S0001870823001895},
		author = {Dmitrii Pirozhkov}}

@incollection {ATV1,
		AUTHOR = {Artin, M. and Tate, J. and Van den Bergh, M.},
		 TITLE = {Some algebras associated to automorphisms of elliptic curves},
	 BOOKTITLE = {The {G}rothendieck {F}estschrift, {V}ol. {I}},
		SERIES = {Progr. Math.},
		VOLUME = {86},
		 PAGES = {33--85},
	 PUBLISHER = {Birkh\"{a}user Boston, Boston, MA},
		  YEAR = {1990},
	   MRCLASS = {14A22 (14H52 16E10 16W50)},
	  MRNUMBER = {1086882},
	MRREVIEWER = {S. Paul Smith},}

@article {Add16,
    AUTHOR = {Addington, Nicolas},
     TITLE = {New derived symmetries of some hyperk\"{a}hler varieties},
   JOURNAL = {Algebr. Geom.},
  FJOURNAL = {Algebraic Geometry},
    VOLUME = {3},
      YEAR = {2016},
    NUMBER = {2},
     PAGES = {223--260},
      ISSN = {2313-1691},
   MRCLASS = {14F05 (14C05 14J28 18E30)},
  MRNUMBER = {3477955},
MRREVIEWER = {Pawel Sosna},
       DOI = {10.14231/AG-2016-011},
       URL = {https://doi.org/10.14231/AG-2016-011},}

@misc{hille2005fouriermukaitransforms,
      title={Fourier-Mukai Transforms},
      author={Lutz Hille and Michel Van den Bergh},
      year={2005},
      eprint={math/0402043},
      archivePrefix={arXiv},
      primaryClass={math.AG},
      url={https://arxiv.org/abs/math/0402043},}

@article{VANDENBERGH1997662,
	  title = {Existence Theorems for Dualizing Complexes over Non-commutative Graded and Filtered Rings},
	  journal = {Journal of Algebra},
	  volume = {195},
	  number = {2},
	  pages = {662-679},
	  year = {1997},
	  issn = {0021-8693},
	  doi = {https://doi.org/10.1006/jabr.1997.7052},
	  url = {https://www.sciencedirect.com/science/article/pii/S0021869397970526},
	  author = {Van den Bergh, Michel},
	  shortauthor = {VdB}}

@misc{denaeghel2005idealclassesdimensionalsklyanin,
      title={Ideal classes of three dimensional Sklyanin algebras}, 
      author={K. De Naeghel and M. Van den Bergh},
	  shortauthor = {NB},
      year={2005},
      eprint={math/0503729},
      archivePrefix={arXiv},
      primaryClass={math.RA},
      url={https://arxiv.org/abs/math/0503729},}

@article {Toen2007,
    AUTHOR = {To\"{e}n, Bertrand},
    TITLE = {The homotopy theory of {$dg$}-categories and derived {M}orita
              theory},
	JOURNAL = {Invent. Math.},
  	FJOURNAL = {Inventiones Mathematicae},
    VOLUME = {167},
	YEAR = {2007},
    NUMBER = {3},
    PAGES = {615--667},
    ISSN = {0020-9910},
	MRCLASS = {18D05 (18E30 18G55 19D55)},
	MRNUMBER = {2276263},
	MRREVIEWER = {Mark Hovey},
	DOI = {10.1007/s00222-006-0025-y},
    URL = {https://doi.org/10.1007/s00222-006-0025-y},}

@article{ARTIN1987171,
	title = {Graded algebras of global dimension 3},
	journal = {Advances in Mathematics},
	volume = {66},
	number = {2},
	pages = {171-216},
	year = {1987},
	issn = {0001-8708},
	doi = {https://doi.org/10.1016/0001-8708(87)90034-X},
	url = {https://www.sciencedirect.com/science/article/pii/000187088790034X},
	author = {Michael Artin and William F Schelter}}

@misc{borisov2024nonexistencephantomsnongenericblowups,
	title={Non-existence of phantoms on some non-generic blowups of the projective plane}, 
	author={Lev Borisov and Kimoi Kemboi},
	year={2024},
	eprint={2405.01683},
	archivePrefix={arXiv},
	primaryClass={math.AG},
	url={https://arxiv.org/abs/2405.01683}, 
}

@article {MR1067406,
    AUTHOR = {Artin, M. and Van den Bergh, M.},
     TITLE = {Twisted homogeneous coordinate rings},
   JOURNAL = {J. Algebra},
  FJOURNAL = {Journal of Algebra},
    VOLUME = {133},
      YEAR = {1990},
    NUMBER = {2},
     PAGES = {249--271},
      ISSN = {0021-8693},
   MRCLASS = {14A22 (16P40 16W50)},
  MRNUMBER = {1067406},
MRREVIEWER = {S. Paul Smith},
       DOI = {10.1016/0021-8693(90)90269-T},
       URL = {https://doi.org/10.1016/0021-8693(90)90269-T},
}

@article {MR3361723,
    AUTHOR = {B\"{o}hning, Christian and Graf von Bothmer, Hans-Christian and
              Katzarkov, Ludmil and Sosna, Pawel},
     TITLE = {Determinantal {B}arlow surfaces and phantom categories},
   JOURNAL = {J. Eur. Math. Soc. (JEMS)},
  FJOURNAL = {Journal of the European Mathematical Society (JEMS)},
    VOLUME = {17},
      YEAR = {2015},
    NUMBER = {7},
     PAGES = {1569--1592},
      ISSN = {1435-9855},
   MRCLASS = {14F05 (14J29 18E30)},
  MRNUMBER = {3361723},
MRREVIEWER = {Jon Eivind Vatne},
       DOI = {10.4171/JEMS/539},
       URL = {https://doi.org/10.4171/JEMS/539},
}

@article {MR4701883,
    AUTHOR = {Krah, Johannes},
     TITLE = {A phantom on a rational surface},
   JOURNAL = {Invent. Math.},
  FJOURNAL = {Inventiones Mathematicae},
    VOLUME = {235},
      YEAR = {2024},
    NUMBER = {3},
     PAGES = {1009--1018},
      ISSN = {0020-9910},
   MRCLASS = {14F08 (14C20 14J26)},
  MRNUMBER = {4701883},
       DOI = {10.1007/s00222-023-01234-0},
       URL = {https://doi.org/10.1007/s00222-023-01234-0},
}

@misc{kuznetsov2009hochschildhomologysemiorthogonaldecompositions,
      title={Hochschild homology and semiorthogonal decompositions}, 
      author={Alexander Kuznetsov},
      year={2009},
      eprint={0904.4330},
      archivePrefix={arXiv},
      primaryClass={math.AG},
      url={https://arxiv.org/abs/0904.4330}, 
}

@article {MR3096524,
    AUTHOR = {Alexeev, Valery and Orlov, Dmitri},
     TITLE = {Derived categories of {B}urniat surfaces and exceptional
              collections},
   JOURNAL = {Math. Ann.},
  FJOURNAL = {Mathematische Annalen},
    VOLUME = {357},
      YEAR = {2013},
    NUMBER = {2},
     PAGES = {743--759},
      ISSN = {0025-5831},
   MRCLASS = {14F05 (14J29)},
  MRNUMBER = {3096524},
MRREVIEWER = {Carlo Giovanni Madonna},
       DOI = {10.1007/s00208-013-0917-2},
       URL = {https://doi.org/10.1007/s00208-013-0917-2},
}

@misc{abdelgadir2014compactmodulinoncommutativeprojective,
      title={Compact moduli of noncommutative projective planes}, 
      author={Tarig Abdelgadir and Shinnosuke Okawa and Kazushi Ueda},
      year={2014},
      eprint={1411.7770},
      archivePrefix={arXiv},
      primaryClass={math.AG},
      url={https://arxiv.org/abs/1411.7770}, 
}

@book{huybrechts2006fourier,
  title={Fourier-Mukai Transforms in Algebraic Geometry},
  author={Huybrechts, D.},
  isbn={9780199296866},
  lccn={2006298244},
  series={Oxford Mathematical Monographs},
  url={https://books.google.co.jp/books?id=9HQTDAAAQBAJ},
  year={2006},
  publisher={Clarendon Press}
}

@incollection {MR2275593,
    AUTHOR = {Keller, Bernhard},
     TITLE = {On differential graded categories},
 BOOKTITLE = {International {C}ongress of {M}athematicians. {V}ol. {II}},
     PAGES = {151--190},
 PUBLISHER = {Eur. Math. Soc., Z\"{u}rich},
      YEAR = {2006},
   MRCLASS = {18E30 (14A22 16D90)},
  MRNUMBER = {2275593},
MRREVIEWER = {Volodymyr V. Lyubashenko},
}

@article{ORLOV201659,
title = {Smooth and proper noncommutative schemes and gluing of DG categories},
journal = {Advances in Mathematics},
volume = {302},
pages = {59-105},
year = {2016},
issn = {0001-8708},
doi = {https://doi.org/10.1016/j.aim.2016.07.014},
url = {https://www.sciencedirect.com/science/article/pii/S0001870816300457},
author = {Dmitri Orlov}
}

@book{Hartshorne,
  title={Algebraic Geometry},
  author={Hartshorne, R.},
  isbn={978-1-4757-3849-0},
  lccn={},
  series={Graduate Texts in Mathematics},
  url={https://link.springer.com/book/10.1007/978-1-4757-3849-0},
  year={1977},
  doi = {https://doi.org/10.1007/978-1-4757-3849-0},
  publisher={Springer New York, NY}
}

@article {MR1258406,
    AUTHOR = {Keller, Bernhard},
     TITLE = {Deriving {DG} categories},
   JOURNAL = {Ann. Sci. \'{E}cole Norm. Sup. (4)},
  FJOURNAL = {Annales Scientifiques de l'\'{E}cole Normale Sup\'{e}rieure. Quatri\`eme
              S\'{e}rie},
    VOLUME = {27},
      YEAR = {1994},
    NUMBER = {1},
     PAGES = {63--102},
      ISSN = {0012-9593},
   MRCLASS = {18E30 (16D90)},
  MRNUMBER = {1258406},
MRREVIEWER = {Jeremy Rickard},
       URL = {http://www.numdam.org/item?id=ASENS_1994_4_27_1_63_0},
}

@article {MR2357345,
    AUTHOR = {Herschend, Martin},
     TITLE = {Tensor products on quiver representations},
   JOURNAL = {J. Pure Appl. Algebra},
  FJOURNAL = {Journal of Pure and Applied Algebra},
    VOLUME = {212},
      YEAR = {2008},
    NUMBER = {2},
     PAGES = {452--469},
      ISSN = {0022-4049},
   MRCLASS = {16G20},
  MRNUMBER = {2357345},
       DOI = {10.1016/j.jpaa.2007.06.004},
       URL = {https://doi.org/10.1016/j.jpaa.2007.06.004},
}

@article {MR2238922,
    AUTHOR = {Lowen, Wendy and Van den Bergh, Michel},
     TITLE = {Deformation theory of abelian categories},
   JOURNAL = {Trans. Amer. Math. Soc.},
  FJOURNAL = {Transactions of the American Mathematical Society},
    VOLUME = {358},
      YEAR = {2006},
    NUMBER = {12},
     PAGES = {5441--5483},
      ISSN = {0002-9947},
   MRCLASS = {18E15 (16S80)},
  MRNUMBER = {2238922},
       DOI = {10.1090/S0002-9947-06-03871-2},
       URL = {https://doi.org/10.1090/S0002-9947-06-03871-2},
}
\end{document}